\documentclass[11pt, reqno]{amsart}

\usepackage{eurosym}
\usepackage{amsfonts}
\usepackage{amsmath,amssymb}
\usepackage[a4paper, total={7.0in, 9.6in}]{geometry}
\usepackage[utf8]{inputenc}
\usepackage[english]{babel}
\usepackage{amsthm}
\usepackage{color,graphics}
\usepackage{upref}
\usepackage[colorlinks,backref]{hyperref}
\usepackage{upref,hyperref,color}

\newtheorem{defi}{Definition}[section]
\newtheorem{thm}{Theorem}[section]
\newtheorem{hyp}{Hypotheses}[section]
\newtheorem{rmq}{Remark}[section]

\newtheorem{lem}{Lemma}[section]

\numberwithin{equation}{section} \allowdisplaybreaks

\title[Stochastic $3^{rd}$-grade fluids equations in $2d$ and $3d$]{On the Existence and long time behaviour of $H^{1}$-Weak Solutions for $2, 3d$-Stochastic $3^{rd}$-Grade Fluids Equations}
\date{\today}

\subjclass[2010]{76A05;	76D03;	60H15; 93D20;  76M35}

\keywords{Third grade fluids, Martingale	solution,	Stochastic PDE,	Stability.}

\thanks{
	This work is funded by national funds through the FCT - Funda\c c\~ao para a Ci\^encia e a Tecnologia, I.P., under the scope of the projects UIDB/00297/2020 and UIDP/00297/2020 (Center for Mathematics and Applications)}

\author[Raya Nouira]{Raya Nouira}
\address[Raya Nouira]{Center for Mathematics and Applications (NovaMath), NOVA	SST,	Portugal}
\email[Raya Nouira]{raya.nouira@gmail.com}

\author[Fernanda Cipriano]{Fernanda Cipriano}
\address[Fernanda Cipriano]{Center for Mathematics and Applications (NovaMath), NOVA SST and Department of Mathematics, NOVA SST, Portugal}
\email[Fernanda Cipriano]{cipriano@fct.unl.pt}

\author[Yassine Tahraoui]{Yassine Tahraoui}
\address[Yassine Tahraoui]{Center for Mathematics and Applications (NovaMath), NOVA SST, Portugal}
\email[Yassine Tahraoui]{tahraouiyacine@yahoo.fr}

\allowdisplaybreaks[1]
\usepackage{comment}

\begin{document}

\begin{abstract}
In the present work, we investigate stochastic third grade fluids equations in a $d$-dimensional setting, for $d = 2, 3$. More precisely, on a bounded and simply connected domain $\mathcal{D}$ of $\mathbb{R}^d$, $d = 2,3$, with a sufficiently regular boundary $\partial \mathcal{D},$ we consider  incompressible third grade fluid equations perturbed by a multiplicative Wiener noise. Supplementing our equations by  Dirichlet boundary conditions and taking initial data in the Sobolev space $H^1(\mathcal{D})$, we establish the existence of global stochastic weak solutions by performing a  strategy based on the conjugation of stochastic compactness criteria and monotonicity techniques. Furthermore, we study the asymptotic behaviour of these solutions, as $t \to \infty$. 

\end{abstract}

\maketitle


\section{Introduction}
\setcounter{equation}{0}

For decades, the motion of Newtonian fluids, modelled by  Navier-Stokes equations, has been studied from mathematical and physical perspectives. Mathematical difficulties in dealing with these equations have been identified, and nowadays theoretical aspects and practical applications are fairly well understood. However, their application in modelling some fundamental physical phenomena such as turbulence still remains a challenge. Although a wide range of fluids in nature exhibit a linear relation between shear stress and shear rate, being classified as Newtonian fluids, a relevant category of natural fluids such as mud, shifting sands, magma, blood plasma and most fluids used in engineering and technology, food processing, the cosmetics industry, etc., do not possess such a linear relation and belong to the class of non-Newtonian fluids. Among these fluids we find the so-called nanofluids which, by enhancing rapid thermal conductivity properties, feature great potential for use in modern and future technologies. Practical studies revealed that third grade fluids equations model quite well the rheological properties of these nanofluid materials, and have recently received special attention from engineers, physicists and mathematicians (see $e.g.$ \cite{hayat2007analytical, hayat2003exact, zhang2023entropy, rasheed2017stabilized, raja2022intelligent} and references therein).

This work is devoted to the mathematical study of stochastic incompressible third grade fluids equations for the random velocity vector field $u$, which reads
\begin{align}
d(u-\alpha_{1}\Delta u) = (&-\nabla p + \mu\Delta u - u\cdot \nabla u + \alpha_{1}\text{div}(u\cdot\nabla A+ (\nabla u)^{T}A+A(\nabla u))\notag\\
&+\alpha_{2}\text{div}(A^{2})+\beta \text{div}(\vert A\vert^{2}A)+\Phi(\cdot,u))dt+ \sigma(\cdot,u)d\mathcal{B}(\cdot), \qquad \text{div} u=0. \label{stoch}
\end{align}
Here $p$ corresponds to the pressure, $\mu>0$ denotes the viscosity of the fluid, $\alpha_{1}$, $\alpha_{2}$ are  normal stress modulli, $\beta>0$ is the specific material modulli and  $ A := A(u)= \nabla u + (\nabla u)^{T}$ represents the Rivlin-Ericksen kinematic tensor $A_1$. Moreover, $\Phi(\cdot, u)$ represents a deterministic external force and $\sigma(\cdot,u)d\mathcal{B}(\cdot)$ stands for the stochastic perturbation, also called the multiplicative Wiener noise. According to Fosdick-Rajagopal \cite{fosdick1980thermodynamics} and Dunn-Rajagopal \cite{dunn1995fluids}, physical and thermodynamic considerations impose the following restrictions on the physical constants
\begin{align}
 \alpha_{1}\geq 0,\quad  \vert \alpha_{1} + \alpha_{2}\vert \leq (24\mu\beta)^{\frac{1}{2}}.\label{Fosdic}
\end{align}

We are faced with a highly non-linear stochastic partial differential equation, in which some of the non-linear terms have derivatives of a higher order than the linear ones, thus bringing major challenges to the analysis.

We recall that the first existence and uniqueness results for local solutions in $\mathbb{R}^{d}$, $d=2,3$, and existence of  global solutions in $\mathbb{R}^{2}$, to deterministic third grade fluids equations ($\sigma = 0$), have been derived by Amrouche and Cioranescu \cite{amrouche1997class}, under small initial data in the Sobolev space $H^3$.
Later on, for initial data in the Sobolev space $H^2$, Busuioc and Iftime \cite{busuioc2004global} showed the global existence and uniqueness of solutions in $\mathbb{R}^2$, and only the existence in $\mathbb{R}^3$. Under the same regularity of the initial data, they also proved in \cite{busuioc2006non} the existence of a global weak solution to a third grade fluid equation supplemented with Navier slip boundary conditions in a bounded domain of $\mathbb{R}^{d},$ $d=2,3$. The uniqueness was obtained only in $\mathbb{R}^2$. Right after, Paicu \cite{paicu2008global} established global $H^1$-weak solutions in $\mathbb{R}^d,$ $d=2,3,$ and showed the weak-strong uniqueness of the solution in $\mathbb{R}^{2}.$

Recently, the authors in \cite{cipriano2021well}, \cite{almeida2020weak}, proved existence and uniqueness results for stochastic third grade fluids equations with multiplicative white noise, on bounded non-axisymmetric domains of $\mathbb{R}^d,$ $d = 2, 3,$ under slip boundary conditions and $H^2$-initial data. In $2d$, a uniqueness type method conjugated with a stopping time argument was applied to show the existence and uniqueness of a strong stochastic solution. In $3d$, a martingale solution (defined on another probability space) was derived by applying stochastic compactness arguments. We emphasize that the above results rely strongly on slip boundary conditions and the $H^2$ regularity of the initial data as they were crucial to perform the analyses; roughly speaking, the slip boundary conditions yield a derivative of order one at the level of boundary terms arising from integration by parts, thus avoiding the creation of strong boundary layers (see $e.g.$ \cite{chemetov2013boundary, chemetov2013inviscid, chemetov2018optimal} and references therein).

Due to its extreme relevance in industrial and technological applications, optimal control of flow should receive special interest. The authors in \cite{tahraoui2023optimal}, \cite{tahraoui2023local} address the problem of deterministic and stochastic optimal control for third grade fluids equations under slip boundary conditions, where the control acts through a distributed mechanical force. In order to control the dynamics, the analysis requires $H^3$ regularity of the solutions. Under the latter regularity, a local well-posedness of solutions was established in \cite{tahraoui2023local}.

In view of martingale solutions, several authors have studied their existence for different models, here we mention, in particular, some works derived in the direction of fluid dynamics (see $e.g.$ \cite{flandoli1995martingale, brzezniak2013existence, chen2019martingale, almeida2020weak} and references therein). Although martingale solutions have been considerably developed, emphasizing mainly for stochastic Navier Stokes equations, it is very important to study this type of solutions for fluids equations of higher complexity, such as third grade fluids equations, as they are technically intractable and require careful analysis.
 
Other relevant aspects to investigate in the context of non-Newtonian flows are the developing turbulence and boundary layer problems, typically associated with less regular solutions. With these questions in mind, this work tackles the existence problems of solutions to the equations \eqref{stoch} on bounded domains of $\mathbb{R}^d$, $d = 2, 3$, for initial data only in $H^1$ and under Dirichlet boundary conditions, as it also addresses the long time behaviour of these solutions. More precisely, throughout this article, we prove the existence of martingale solutions with sample paths in $H^1$. In addition, we analyse the asymptotic behaviour of the latter, as $t \to \infty$; namely, under certain conditions on the force and diffusion coefficients and assuming that the viscosity is large enough (a sufficiently small Reynolds number), we prove the exponential decay of $H^1$ stochastic weak solutions to the stationary trivial equilibrium in the mean square and almost surely. Thus recalling the importance behind studying the long time behaviour of fluids flows. 

It is worth stressing that the methods applied here to show existence results are very different from the methods used in previous articles. For initial data in $H^1$, we do not expect to have uniqueness, so the strategies applied previously will not be useful, thus more involved analyses are required. Additionally, the a priori estimates of the solutions hold only in $H^1$ and $W^{1,4}$, which turns out to be insufficient to pass to the limit in the equations solely relying on stochastic compactness arguments. To overcome the difficulty stemming from the low regularity of the initial data, we will take advantage of the monotonicity property of some non-linear terms, that can be written as a non-linear monotone operator, see \cite{paicu2008global, busuioc2008steady}. The remaining terms are manipulated separately through stochastic compactness criteria. Therefore, our strategy is based on an appropriate conjugation of stochastic compactness arguments to pass to the limit the equation in the weak sense, and on the application of monotonicity techniques to identify the limit defined on the new probability space. On the other hand, to prove the stability results, we follow methods similar to those used in $e.g.$ \cite{caraballo2002exponential, cipriano2019asymptotic}. Furthermore, we discuss, in particular, the stability properties of the stationary solution, with certain regularities, to deterministic third grade fluids equations of steady motion type.

The remainder of the article is structured as follows, Section 2 consists of a self-completed work set-up, where we give the main system to be studied, then introduce appropriate spaces, preliminary results and definitions, and necessary assumptions for our analysis. Section 3 is entirely devoted to the main result, $i.e.$ the existence of  martingale solutions to the considered problem. This section is quite long, but self-contained, as it is divided into six subsections, in the first one, we state the main result, implement the method of Galerkin approximations along with cut-off techniques and prove the a priori estimates. The second subsection provides convenient uniform estimates, and in the third we apply the stochastic compactness criterion including the tightness property and Prokhorov's theorem. The fourth subsection consists in applying Skorokhod's theorem and establishing convergence results. In the fifth subsection, we introduce the so-called monotonicity property. The last subsection deals with the last step of the proof of the main result, where we perform the limit passage. Finally, section 4 studies the asymptotic behaviour of martingale solutions, as $t\to\infty.$

\begin{center}
\section{Notations and preliminaries}
\setcounter{equation}{0}
\end{center}

Let $\mathcal{D}$ be a bounded and simply connected domain of the space $\mathbb{R}^{d}$, $d=2,3$ with sufficiently regular boundary $\partial{\mathcal{D}}$. For a finite horizon $T >0$, let $(\Omega, \mathcal{F}, \{\mathcal{F}_t\}, \mathbb{P})$ be a stochastic basis, with $\{\mathcal{F}_t\} := \{\mathcal{F}_{t}\}_{t\in[0,T]}$ being a normal filtration $i.e.$ a filtration satisfying the usual conditions.

Consider the following stochastic third grade fluids system supplemented with Dirichlet boundary conditions
\begin{equation}
\left\{
\begin{array}{cccc}
d(U-\alpha_{1}\Delta U)= \bigl(-\nabla p+ \mu\Delta U - U\cdot \nabla U + \alpha_{1}\mathrm{div}(U \cdot \nabla A+(\nabla U)^{T}A +A(\nabla U)) \vspace{2mm} \\
		\vspace{2mm}
		+\alpha_{2} \mathrm{div}(A^{2})+\beta{\rm div}\left(|A|^2 A\right)+ \Phi(\cdot,U) \bigr)dt
		+ \sigma(\cdot,U) d\mathcal{B}(t)
		&  \multicolumn{1}{l}{\mbox{in}\
			\mathcal{D}\times\Omega\times (0,T),} \vspace{2mm}\\
			\multicolumn{1}{l}{\mathrm{div}\,U=0} & \multicolumn{1}{l}{\mbox{in}\
			\mathcal{D}\times \Omega\times (0,T),} \vspace{2mm}\\
			\multicolumn{1}{l}{U=0} & \multicolumn{1}{l}{\mbox{on}\
			\partial{\mathcal{D}}\times \Omega \times (0,T),}\vspace{2mm} \\
			\multicolumn{1}{l}{U(0)=U_{0}} & \multicolumn{1}{l}{\mbox{in}\ \mathcal{D}\times \Omega,}
\end{array}
\right.\label{syst}
\end{equation}
where $U =(U^{i})_{i=1}^{d},$ $d=2,3$, is the velocity vector field, such that for any arbitrary $t\in [0,T],$ the random variable $U(t)$ is $\{\mathcal{F}_t\}$-measurable,  $U_0$ is the initial data, $\sigma(\cdot,U)$ denotes the diffusion coefficient and $(\mathcal{B}(t))_{t \in [0,T]}$ is a cylindrical Wiener process with respect to the filtration $\{\mathcal{F}_t\}_{t \in [0,T]}$.
\vspace{0.2cm}
\subsection{General framework and related lemmas}
For convenience, we do not distinguish between notations for scalars, vectors, and matrices when the context is clear. We omit the notation for the omega dependence, since it is obvious that we are working in a random framework. When not necessary, we also drop the notations for time and space dependence.\\

For $d=2,3$, we denote the usual scalar product of vector fields $u, y\in \mathbb{R}^{d}$ by $u\cdot y= \sum_{i=1}^{d}u^{i}y^{i},$ and of matrices $M = (M^{ij})_{i,j}$, $L=(L^{ij})_{i,j} \in \mathcal{M}_{d\times d}(\mathbb{R}^d)$ by $M\cdot L= \sum_{i,j=1}^{d}M^{ij}L^{ij}.$ Moreover, we set the notations $\vert M \vert^{2} = M\cdot M,$ $M^{2}= (\sum_{k=1}^{d}M^{ik}M^{kj})_{i,j=1}^{d},$ and we define the divergence of a matrix $M$ as follows $(\text{div}(M))_{i}=\sum_{j=1}^{d}\partial_{j} M^{ij}.$ \vspace{0.2cm}

Given the bounded domain $\mathcal{D}\subset \mathbb{R}^{d}$, $d=2,3$, let $L^{p}(\mathcal{D})$, $p\geq1$, be the Lebesgue space endowed with the norm $\|\cdot\|_{p},$ $W^{r,p}(\mathcal{D}),$ $r \in \mathbb{N}$, $p\geq 1$, be the Sobolev spaces endowed with the norm $\|\cdot\|_{W^{r,p}}$ and $H^{r}(\mathcal{D}),$ $r\in \mathbb{N}$, be the Sobolev spaces $W^{r,p}(\mathcal{D}),$ for $p = 2$, endowed with the norm $\|\cdot\|_{H^{r}}$. In order to keep notations simple for spaces in $d=2,3$ dimension, we set out the following, $\textbf{L}^{p}(\mathcal{D}):=( L^{p}(\mathcal{D}))^d,$ $\textbf{W}^{r,p}(\mathcal{D}):= (W^{r,p}(\mathcal{D}))^d$ and $\textbf{H}^r(\mathcal{D}) := (H^r(\mathcal{D}))^d.$ Note that the notations for norms remain the same in $d= 2,3$ dimension.\vspace{0.2cm}

Let us denote by $\mathcal{V}$ the space of $\mathcal{C}^\infty (\mathcal{D})$ vector fields $u$ with compact support, satisfying $\text{div} u =0$ in $\mathcal{D}$. Define $H$ and $V$ to be the closure of $\mathcal{V}$ in $\textbf{L}^{2}(\mathcal{D})$ and $\mathbf{H}^{1}_0(\mathcal{D})$ respectively, that is 
\begin{align*}
&H = \lbrace u \in \mathbf{L}^{2}(\mathcal{D}) : \mathrm{div}u = 0 \phantom{x} \text{in} \phantom{x} \mathcal{D} \phantom{x} \text{and} \phantom{x} u\cdot n=0 \phantom{x} \text{on} \phantom{x} \partial\mathcal{D}\rbrace; \\ &
V = \lbrace u \in \mathbf{H}^{1}_{0}(\mathcal{D}) : \mathrm{div}u = 0 \phantom{x} \text{in} \phantom{x} \mathcal{D}\rbrace.
\end{align*}
$H$ denotes the Hilbert space endowed with the usual $L^{2}$-inner product $(\cdot,\cdot)$ together with the associated norm $\|\cdot\|_{2},$ and $V\subset \mathbf{H}^{1}(\mathcal{D})$ is the Hilbert space endowed with the following inner product
\begin{equation}
(u,y)_{V}:= (u-\alpha_{1}\Delta u,y) = (u,y)+\alpha_{1}(\nabla u,\nabla y), \phantom{xx} \text{for any} \phantom{x} u,y \in V,\label{innerV}
\end{equation}
where $\alpha_1> 0,$ and the associated norm $\|\cdot \|_{V}$ being equivalent to the norm in $\mathbf{H}^1(\mathcal{D}),$ $\|\cdot\|_{H^{1}}$. We hereby recall the following
\begin{equation*}
(u,y) = \int_{\mathcal{D}}u \cdot y dx, \quad \forall u,y \in H, \qquad (M,L) = \int_{\mathcal{D}} M \cdot L dx, \quad \forall M, L \in \mathcal{M}_{d \times d}(H).
\end{equation*}

Let $E$ denote a real Banach space endowed with the norm $\|\cdot\|_{E}$. We define for $d=2,3$
$$\mathbf{E}:= (E)^d =\{(h^1,\cdots,h^d), \quad h^l\in E, \quad l=1,\cdots,d\},$$
and $\|\cdot\|_{E}$ is understood as follows
\begin{align*}
&\| h\|_E^2= \| h^1\|_E^2+\cdots+\| h^d\|_E^2 \quad \text{for any} \quad h=(h^1,\cdots,h^d) \in \mathbf{E}, \notag \\ 
&\| h\|_{E}^2= \sum_{i,j=1}^d\| h^{ij}\|_E^2 \quad \text{for any} \quad h\in \mathcal{M}_{d\times d}(E).
\end{align*} 

In addition, let $L^{p}(0,T;\mathbf{E}),$ $p\geq1$, be the space of measurable and $p$-integrable functions $u$ defined on $[0,T]$ with values in $\mathbf{E}$. Denote by $\mathcal{C}^{0,\alpha}([0,T], \mathbf{E}),$ $\alpha \in (0,1),$ the space of $\alpha$-Hölder continuous functions defined on $[0,T]$ with values in $\mathbf{E}$ and endowed with the norm
\begin{equation}
\|u\|_{\mathcal{C}^{0,\alpha}([0,T],E)}=\sup_{t\in [0,T]}\|u(t)\|_{E} + \sup_{t,s \in [0,T], t\neq s} \dfrac{\|u(t)-u(s)\|_{E}}{\vert t-s \vert^{\alpha}},\label{normHolder}
\end{equation}
and let $W^{\gamma,q}(0,T;\mathbf{E})$, $\gamma \in (0,1)$, $q\in (1,\infty)$, be the fractional Sobolev space endowed with the norm
\begin{equation}
\|u\|_{W^{\gamma,q}(0,T;E)}^{q} = \|u\|_{L^{q}(0,T;E)}^{q}+\int_{0}^{T}\int_{0}^{T}\dfrac{\|u(t)-u(s)\|_{E}^{q}}{\vert t-s\vert^{1+\gamma q}}dsdt.\label{normW}
\end{equation}

\begin{lem}[\cite{simon1986compact}]
Consider the Banach spaces $\mathbf{E}_{0}$, $\mathbf{E}_{1}$ and $\mathbf{E}$ such that $\mathbf{E}_0\subset \mathbf{E}\subset \mathbf{E}_1$. Assume that the first inclusion is compact. Then, for $\alpha\in(0,1)$ and $p \in (1,\infty)$, the space 
$$L^p(0,T; \mathbf{E}_0)\cap \mathcal{C}^{0,\alpha}([0,T],\mathbf{E}_1),$$
\label{C.2}
with the norm 
\begin{equation}
\|u\|_{\alpha,p, E_0, E_1}:=\|u\|_{L^p(0,T; E_0)}+\|u\|_{\mathcal{C}^{0,\alpha}([0,T],E_1)}\label{norm.com}
\end{equation}
is compactly embedded in the space $C([0,T], \mathbf{E}).$
\end{lem}

Define the trilinear form $b(\cdot, \cdot, \cdot) : V \times V \times V \to \mathbb{R}$ by
\begin{align}
b(u,y,z)& = (u\cdot \nabla y, z).\label{trilin}
\end{align}
Since $u \in V$, then a standard integration by parts yields
\begin{equation}
b(u,y,z)= -b(u,z,y), \qquad \forall u,y,z \in V.\label{trilin2}
\end{equation}

Korn's and Poincaré inequalities and the embeddings $V \hookrightarrow \mathbf{L}^4(\mathcal{D}) \hookrightarrow \mathbf{L}^2(\mathcal{D})$, yield the following lemma.
\begin{lem}
There exists a positive constant $C_*$, such that 
\begin{equation}
\|u\|_{W^{1,4}} \leq C_*\|A(u)\|_{4}, \quad \forall u \in V. \label{KORN*}
\end{equation}
\end{lem}
In the sequel and throughout the paper, $C$ will denote a positive generic constant that may change from line to line in accordance with the various inequalities and computations.

\begin{hyp}
\label{Hyp1}
We present here assumptions on the data of the system \eqref{syst}. Let $2<p<\infty$.
\begin{enumerate}
\item[$(H_{1})$] The initial condition, $U_{0}: \Omega \rightarrow V,$ is $\mathcal{F}_{0}$-measurable and belongs to the space $L^{p}(\Omega,V),$ i.e. $\mathbb{E}\|U_{0}\|_{V}^{p}< \infty.$
\item[$(H_{2})$] There exist two constants $\kappa, \ell> 0$ such that, the force $\Phi : [0,T] \times V \to \mathbf{L}^{2}(\mathcal{D})$ satisfies 
\begin{align}
&\|\Phi(t,u) - \Phi(t,y)\|_{2}^{2}\leq \kappa\|u-y\|_{2}^{2}, \quad \forall u,y \in V, \phantom{x} t \in [0,T],\label{lipschitz1}
\\ &
\|\Phi(t,u)\|_{2}^{2}\leq \ell(1+\|u\|_{V}^{2}), \quad \forall u \in  V, \phantom{x} t \in [0,T].\label{growth1}
\end{align}
\item[$(H_{3})$] The viscosity $\mu$, the normal stress modulli $\alpha_{1}$, $\alpha_{2}$ and the material modulli $\beta$ satisfy the conditions
\begin{equation}
\mu, \beta \geq 0 \quad \text{and} \quad 3\alpha_{1}^{2} + 4(\alpha_{1}+\alpha_{2})^{2}\leq 24\mu\beta. \label{RESTMON}
\end{equation}
\end{enumerate}
Note that the restrictions \eqref{RESTMON} are strongly related to monotonicity techniques.
\end{hyp}

Consider the generalized Stokes problem with the Dirichlet boundary condition
\begin{equation}
\begin{cases}
\tilde{f}-\alpha_{1}\Delta \tilde{f}+\nabla p = f, \phantom{x}  \text{div}\tilde{f}=0 \phantom{x} \text{in} \phantom{x} \mathcal{D},\\
\tilde{f} = 0 \phantom{x} \text{on} \phantom{x} \partial \mathcal{D}.\label{Stokes1}
\end{cases}
\end{equation}
According to \cite{temam2001navier} (see also \cite{razafimandimby2010weak}), for any $f \in \mathbf{H}^{m}(\mathcal{D})$, $m\geq 0,$ the problem \eqref{Stokes1} has a unique solution $(\tilde{f},p) \in \mathbf{H}^{m+2}(\mathcal{D})\times \mathbf{H}^{m+1}(\mathcal{D}),$ with $p$ unique up to a constant, satisfying 
\begin{equation}
\|\tilde{f}\|_{H^{m+2}}\leq C\|f\|_{H^{m}}.\label{Stokes1ineq}
\end{equation}
In addition, by the definition of the inner product \eqref{innerV}, we get
\begin{equation}
(\tilde{f},g)_{V}=(f,g), \phantom{xx} \text{for any} \phantom{x} g\in V. \label{Stokes1eq}
\end{equation}

Furthermore, we state another result on the existence and uniqueness of a solution to the Dirichlet problem \eqref{Stokes1} that will play a crucial role in the proof of uniform estimates (see Lemma \ref{lemma.4}). The following Lemma is a particular case of \cite[Theorem 7]{kim2009existence}. Its proof can be performed using slight adaptations in the proof of the referred theorem.   
\begin{lem}
Let $\mathcal{D}$ be a bounded open domain of $\mathbb{R}^{d}$, $d=2,3$, with sufficiently regular boundary $\partial\mathcal{D}$. Then, for every $f \in (V\cap \mathbf{H}^{2})^{*},$ there exists a unique very weak solution $\tilde{f} \in \mathbf{L}^{2}(\mathcal{D})$ to the Stokes problem \eqref{Stokes1}. Moreover, we have for a positive constant $ C $
\begin{equation}
\|\tilde{f}\|_{2} \leq C \|f\|_{(V \cap H^{2})^{*}}. \label{Stokes2ineq}
\end{equation}
\end{lem}

\subsection{Stochastic framework} \label{Subsection2.2}
Recall the complete probability space $(\Omega,\mathcal{F},\mathbb{P})$ and set $\Omega_{T}:=\Omega \times [0,T]$. Denote by $\mathbb{E}$ the mathematical expectation under the probability measure $\mathbb{P}$ and by $\mathcal{P}_{T}$ the predictable $\sigma$-algebra $i.e.$ the $\sigma$-algebra generated by $\{\mathcal{F}_{t}\}$-adapted and left continuous stochastic processes $X:\Omega_{T}\to \mathbf{E}$. Recall that the process $X$ is said to be $\mathbf{E}$-predictable if it is $\mathcal{P}_{T}/\mathfrak{B}(\mathbf{E})$-measurable.\vspace{0.2cm}

Let $L^{p}(\Omega, L^{m}(0,T;\mathbf{E}))$, $p,m \geq 1$, be the space of $\mathbf{E}$-predictable processes $U = U(\omega,t)$ defined on $\Omega$ with values in $L^m(0,T,\mathbf{E})$, and  endowed with the norms
\begin{equation*}
\| U\|_{L^{p}(\Omega,L^{m}(0,T;E))}=\left(\mathbb{E}\left(\int_{0}^{T}\| U\|_{E}^{m}dt\right)^{\frac{p}{m}}\right)^{\frac{1}{p}}
\end{equation*}
and if $m = \infty$,
\begin{equation*}
\|U\|_{L^{p}(\Omega,L^{\infty}(0,T;E))}=\left(\mathbb{E}\sup_{t\in [0,T]}\|U\|_{E}^{p} \right)^{\frac{1}{p}}.
\end{equation*}
In addition, let $L^{p}(\Omega,L^{\infty}_{w}(0,T;\mathbf{E})),$ $p\geq 1$, denote the space of $\mathbf{E}$-predictable and weakly-* measurable processes $U$ defined on $\Omega$ with values in $L^{\infty}(0,T;\mathbf{E})$, satisfying $ \mathbb{E}\|U\|_{L^{\infty}(0,T;E)}^{p} < \infty $.

\subsubsection{The Wiener process}
Let $\mathcal{U}_{0}$ be a separable Hilbert space, and $\lbrace e_{k}\rbrace_{k\in \mathbb{N}^{+}}$ be an orthonormal basis of $\mathcal{U}_{0}$. Consider an $\{\mathcal{F}_{t}\}$-cylindrical Wiener process $\mathcal{B}(t),$ $t \in [0,T],$ with values in $\mathcal{U}_{0},$ defined as follows 
\begin{equation}
\mathcal{B}(t) = \sum_{k \in \mathbb{N}^{+}}\beta_{k}(t)e_{k}, \quad t \in [0,T] \label{InfiniteB}
\end{equation}
where $(\beta_{k}(t))_{k\in \mathbb{N}^{+}}$ is a sequence of independent real valued $\{\mathcal{F}_{t}\}$-standard Wiener processes. The series does not converge in $\mathcal{U}_{0}.$ However, we still can easily represent the stochastic process $\mathcal{B}$ defined in \eqref{InfiniteB} as a $Q$-Wiener process with values in a complete larger Hilbert space. We write
\begin{equation}
\mathcal{B}(t) = \sum_{k \in \mathbb{N}^{+}}\beta_{k}(t)e_{k}= \sum_{k \in \mathbb{N}^{+}}\sqrt{\lambda_{k}}\beta_{k}(t)g_{k}, \quad t \in [0,T]
\end{equation}
thus, $\mathcal{B}(t),$ $t \in [0,T],$ is a $Q$-Wiener process with covariance $Q$ (a symmetric, trace class non-negative operator), such that $Qg_{k}=\lambda_{k}g_{k},$ $\lambda_{k}> 0,$ $k \in \mathbb{N}^{+},$ where $(\lambda_{k}= k^{-2})_{k\in \mathbb{N}^{+}}$ is a sequence of eigenvalues of $Q$ corresponding to the eigenvectors $\lbrace g_{k} = ke_{k}\rbrace_{k \in \mathbb{N}^{+}}$. The set $\lbrace g_{k}\rbrace_{k\in \mathbb{N}^{+}}$ forms an orthonormal basis of the larger space, given by $\mathcal{U}_{1}:= Q^{-1/2}(\mathcal{U}_{0}),$ that defines a Hilbert space when endowed with the inner product 
\begin{equation*}
(g,h)_{\mathcal{U}_{1}}:=(Q^{1/2}g,Q^{1/2}h)_{\mathcal{U}_{0}}, \quad \forall g,h \in \mathcal{U}_{1}.
\end{equation*}

Therefore the series in \eqref{InfiniteB} converges in $L^{2}(\Omega,C([0,T],\mathcal{U}_{1}))$, and $\mathbb{P}$-a.s., the trajectories of $\mathcal{B}(t),$ $t \in [0,T],$ are in the space $C([0,T],\mathcal{U}_{1}).$ 
\subsubsection{The diffusion coefficient and the stochastic integral}
We give appropriate assumptions on the diffusion coefficient that will allow us to define the Itô integral.

Given the separable Hilbert space $\mathbf{L}^{2}(\mathcal{D}),$ endowed with the norm $\|\cdot\|_{2}$, we consider the space of all Hilbert-Schmidt operators from $\mathcal{U}_{0}$ into $\mathbf{L}^{2}(\mathcal{D})$ defined as follows
\begin{equation*}
L_2^0:= L_{2}(\mathcal{U}_{0},\mathbf{L}^{2}(\mathcal{D})):= \lbrace \sigma : \mathcal{U}_{0} \rightarrow \mathbf{L}^{2}(\mathcal{D}): \|\sigma\|_{L_{2}(\mathcal{U}_{0},\mathbf{L}^{2}(\mathcal{D}))}^{2}= \sum_{k=1}^{\infty} \|\sigma e_{k}\|_{2}^{2}< \infty\rbrace.
\end{equation*}
The latter is a separable Hilbert space endowed with the norm $\|\cdot\|_{L_{2}^{0}}:= \|\cdot\|_{L_{2}(\mathcal{U}_{0},\mathbf{L}^{2}(\mathcal{D}))}^{2}.$ 

Let $\sigma : \Omega_{T} \rightarrow L_{2}^{0}$ be a measurable mapping from $(\Omega_T, \mathcal{P}_{T})$ into $(L_{2}^{0},\mathfrak{B}(L_{2}^{0})),$ that satisfies the following
\begin{equation*}
\mathbb{E}\int_{0}^{T}\|\sigma(t)\|^{2}_{L_{2}^{0}}dt< \infty.
\end{equation*} 
Namely, let us consider an $L_{2}^{0}$-predictable process $\sigma \in L^{2}(\Omega_{T};L_{2}^{0})$ and set $\sigma^{k}=\sigma e_{k}$, then the stochastic integral with respect to the cylindrical Wiener process $\mathcal{B}$ given by
\begin{equation*}
\int_{0}^{t}\sigma(s) d\mathcal{B}(s) = \sum_{k=1}^{\infty}\int_{0}^{t}\sigma^{k}(s) d\beta_{k}(s), \quad \forall t \in [0,T]
\end{equation*}
is well defined. 
\begin{hyp}
\label{Hyp2}
Consider a family of functions 
$\sigma^k : (t,u)\in[0,T]\times V \to \mathbf{L}^2(\mathcal{D})\ni \sigma^k(t,u),$ $k \in \mathbb{N}^+,$
where, $t\mapsto \sigma^k(t, u),$ $k\in \mathbb{N}^+,$ is measurable for all $u \in V,$ and there exist two constants $\kappa, \ell> 0$, such that  
\begin{align}
&\sum_{k=1}^{\infty}\|\sigma^{k}(\cdot,u)-\sigma^{k}(\cdot,y)\|_{2}^{2}\leq \kappa \|u-y\|_{2}^{2}, \quad \forall u,y\in V, \phantom{x} a.e. \phantom{x}t\in [0,T],\label{lipschitz2}
\\ &
\sum_{k=1}^{\infty}\|\sigma^{k}(\cdot,u)\|_{2}^{2}\leq \ell(1+\|u\|_{V}^{2}), \quad \forall u \in V, \phantom{x} a.e. \phantom{x}t\in [0,T].\label{growth2}
\end{align}
\end{hyp}
Now, let us define for all $t \in [0,T]$ and $u \in V$, the linear mapping $\sigma(t,u) : \mathcal{U}_{0}\rightarrow \mathbf{L}^{2}(\mathcal{D}),$ as follows $\sigma (t,u) e_{k}=\sigma^{k}(t,u),$ for all $k \in \mathbb{N}^{+}.$
From the above hypotheses, we deduce that the diffusion coefficient $\sigma$ is $\kappa$-Lipschitz continuous in $u$, uniformly in $t$, and satisfies a linear growth condition.
Therefore, for any $t\in [0,T]$, and $u \in V$, the process $\sigma(t,u)$ is a Hilbert-Schmidt operator from $\mathcal{U}_{0}$ into $\mathbf{L}^{2}(\mathcal{D})$, that is
\begin{equation*}
\sigma : [0,T]\times V \rightarrow L_{2}(\mathcal{U}_{0},\mathbf{L}^{2}(\mathcal{D})).
\end{equation*}
Moreover, for any $V$-predictable process $U \in L^{2}(\Omega_T;V)$, the diffusion coefficient $\sigma(\cdot,U)$ is $L_{2}^{0}$-predictable, and $\sigma \in L^{2}(\Omega_{T}\times V; L_{2}^{0}).$ Taking into account the above discussion, we infer that the stochastic integral
\begin{equation}
\int_{0}^{t}\sigma(s,U)d\mathcal{B}(s) = \sum_{k=1}^{\infty}\int_{0}^{t}\sigma^{k}(s,U)d\beta_{k}(s), \quad \forall t\in [0,T]
\end{equation}
is well defined and it is a continuous square integrable $\mathbf{L}^{2}(\mathcal{D})$-valued martingale with respect to $\{\mathcal{F}_{t}\}_{t \in [0,T]}$. We refer to \cite{da2014stochastic, liu2015stochastic}, for more details about the above setting.\vspace{0.2cm}

We close this section by introducing the notion of weak probabilistic solution, also called martingale solution.
\begin{defi}
\label{weak.sol}
Assume the Hypotheses \ref{Hyp1}-\ref{Hyp2}. A system 
$(\widetilde{\Omega},\widetilde{\mathcal{F}}, \widetilde{\mathbb{P}},\{\widetilde{\mathcal{F}}_t\}_{t\in[0,T]}, \widetilde{\mathcal{B}}, \widetilde{U})$
is said to be a stochastic weak (or martingale) solution of the  system \eqref{syst} if 
\begin{enumerate}
\item $\mathcal{L}(\widetilde{U}(0))=\mathcal{L}(U_{0}),$ where $\mathcal{L}(\cdot)$ stands for the law of a random variable,
\item $(\widetilde{\Omega},\widetilde{\mathcal{F}}, \widetilde{\mathbb{P}})$ is a complete probability space,
\item  $\widetilde{\mathcal{B}}$ is a Wiener process defined on the probability space $(\widetilde{\Omega},\widetilde{\mathcal{F}}, \widetilde{\mathbb{P}})$ for the filtration $\{\widetilde{\mathcal{F}}_t\}_{t\in[0,T]}$,
\item $\widetilde{U}$ is $V$-predictable and belongs to the space $$L^{p}(\widetilde{\Omega},C([0,T],H))\cap L^{p}(\widetilde{\Omega} ,L^{\infty}_{w}(0,T;V))\cap L^{4}(\widetilde{\Omega}, L^{4}(0,T; \mathbf{W}^{1,4}(\mathcal{D}))),$$
\item the following equation holds $\widetilde{\mathbb{P}}-$a.s., for all $t\in [0,T]$
\begin{align}
\notag (\widetilde{U}(t),\varphi-\alpha_{1}\Delta\varphi) &= (\widetilde{U}_{0},\varphi-\alpha_{1}\Delta\varphi) -\int_{0}^{t}\int_{\mathcal{D}}(\widetilde{U}\cdot \nabla \widetilde{U})\cdot \varphi dxds +\int_{0}^{t}\int_{\mathcal{D}} \alpha_{1}\sum_{i,j,l=1}^{d}\widetilde{U}^{i}\widetilde{A}^{lj}\cdot \partial_{i}\partial_{j}\varphi^{l} dxds\\ \notag& 
-\int_{0}^{t}\int_{\mathcal{D}}\left(\mu \nabla \widetilde{U} + \alpha_{1}((\nabla \widetilde{U})^{T}\widetilde{A} + \widetilde{A}(\nabla \widetilde{U}))+\beta
(\vert \widetilde{A}\vert^{2}\widetilde{A})+\alpha_{2}\widetilde{A}^{2}\right)\cdot \nabla \varphi dxds
\notag\\
&
+\int_{0}^{t}\left( \Phi(s,\widetilde{U}(s)),\varphi \right)\,ds+\int_{0}^{t}\left(  \sigma(s,\widetilde{U}(s)),\varphi \right) d\widetilde{\mathcal{B}}(s), 
\label{weak.stoch}
\end{align}
for all $\varphi \in \mathcal{V}$, where $\widetilde{A}:=A(\widetilde{U}).$ 
\end{enumerate}
\end{defi}

\begin{center}
\section{Existence of stochastic weak solutions}\label{section3}
\end{center}
\setcounter{equation}{0}
The main result of this section is given as follows.
\begin{thm}
\label{main.result}
Assume that Hypotheses \ref{Hyp1}-\ref{Hyp2} hold. Then, there exists a martingale solution $$(\widetilde{\Omega}, \widetilde{\mathcal{F}},\widetilde{\mathbb{P}}, \lbrace\widetilde{\mathcal{F}_{t}}\rbrace_{t \in [0,T]}, \widetilde{\mathcal{B}},\widetilde{U})$$ to the system \eqref{syst} according to Definition \ref{weak.sol}.
\end{thm}

As mentioned above, to prove Theorem \ref{main.result}, we first apply the finite-dimensional Galerkin approximations method and use cut-off techniques for non-linear terms in order to establish both the existence of a global solution to the approximated system and a priori estimates.

\subsection{Approximations and a priori estimates}
The embedding $V \hookrightarrow H$ being compact, then there exists a basis $\lbrace \text{v}_{j}\rbrace_{j\in \mathbb{N}} \subset V$ such that
\begin{equation}
(u,\text{v}_{j})_{V}=\lambda_{j}(u,\text{v}_{j}), \qquad \forall u \in V, \quad j \in \mathbb{N}.\label{basis}
\end{equation}
The sequence of eigenvectors $\lbrace \text{v}_{j}\rbrace_{j\in \mathbb{N}}$ is an orthonormal basis in $H$ and an orthogonal basis in $V$, and the corresponding sequence of eigenvalues $( \lambda_{j})_{j\in \mathbb{N}}$ satisfies $\lambda_{j}>0, \forall j \in \mathbb{N},$ \text{and} $\lambda_{j}\to \infty$ \text{as} $j \to \infty.$ Note that the problem \eqref{basis} is elliptic and its solution is $C^{\infty}$-regular, therefore we can assume $H^{4}(\mathcal{D})$ regularity $i.e.$ $\text{v}_{j} \subset \mathbf{H}^{4}(\mathcal{D})$, $j\in \mathbb{N}$, when needed, see $e.g.$ \cite{busuioc2003second}.

Consider the finite dimensional space $V_{n} = \mathrm{span}\lbrace \text{v}_{1},\dots,\text{v}_{n}\rbrace$ 
and set $W:=\mathbf{H}^{4}(\mathcal{D}) \cap V$. Define the operator $P_n$ by
\begin{equation*}
P_n : W^{*} \to V_n; \quad u \mapsto P_n u = \sum_{j=1}^{n} (u,\text{v}_{j})_{W^{*},W}\text{v}_{j}.
\end{equation*}
\begin{rmq}
\label{lem-proj-H}
The restriction	of $P_n$ to $H$,	denoted similarly, is	the	$(\cdot,\cdot)$-orthogonal projection given by
\begin{equation*}
P_n : H\to V_n; \quad u\mapsto P_n u=\sum_{j=1}^n (	u,\text{v}_j) \text{v}_j.
\end{equation*}
It is clear	to notice that $\|P_n u\|_2 \leq \|u\|_2$	 for any	$u\in H,$ and $P_n$ is self-adjoint.
\end{rmq}

Let $N>0$ and consider a family of smooth functions $\iota_N:[0,\infty) \to [0,1],$ satisfying
\begin{align}\label{cut-function}
\iota_{N}(x)=\begin{cases}
1, \quad 0\leq x\leq N,\\[0.15cm]
0, \quad  2N \leq x.	\end{cases}
\end{align}
Moreover, let us denote by $\phi_N$ the family of functions defined on $W^*$, given as follows 
\begin{equation}
\phi_N(u)=\iota_N (\|u\|_{W^*}), \quad \forall u \in  W^*. \label{cutoff}
\end{equation}

Now, we write the Galerkin approximations of system \eqref{syst}
\begin{equation}
\left\{
\begin{array}{l}
		d\left( U_n -\alpha_{1}P_n\Delta U_{n},\varphi \right) =\bigl( \mu P_n\Delta U_{n}-P_n U_{n}\cdot \nabla U_n + \alpha_1P_n\mathrm{div}\left((\nabla U_{n})^{T}A_n + A_{n}(\nabla U_{n})\right)\vspace{2mm} \\
		 
		\qquad\qquad\qquad\qquad +\alpha_{1}P_n\mathrm{div}(U_{n}\cdot\nabla A_{n})+\alpha_{2}P_n\mathrm{div}(A_{n}^{2}) + \beta P_n{\rm div}\left(|A_n|^2A_n\right)
\vspace{2mm} \\
		\qquad\qquad\qquad\qquad+P_n\Phi(\cdot,U_{n}), \varphi \bigr)dt
		+\left( P_n\sigma(\cdot,U_{n}),\varphi \right) \,d\mathcal{B}(t),\qquad  \forall \varphi \in	V_{n},
		\vspace{2mm} \\
		U_{n}(0)=U_{n,0},
	\end{array}
	\right. \label{Galerkin.syst}
\end{equation}
where $A_{n}=\nabla U_{n} + \nabla U_{n}^{T},$ $U_{n,0}$ denotes the projection of the initial condition $U_{0}$ onto the space $V_{n}$ and $U_{n}$ is uniquely given by the linear combinations of the basis $\lbrace \text{v}_{j}\rbrace_{j\in \mathbb{N}}$, $U_{n}(t)=\sum_{j=1}^{n}c_{j}^{n}(t)\text{v}_{j}.$

From the elliptic problem \eqref{basis}, we can derive an orthonormal basis for $V$, given as follows $\lbrace \hat{\text{v}}_{j} = \frac{1}{\sqrt{\lambda}_{j}}\text{v}_{j}\rbrace_{j\in \mathbb{N}},$ and
$U_{n,0}=\sum_{j=1}^{n}\left(U_{0},\text{v}_{j}\right)\text{v}_{j}=\sum_{j=1}^{n}\left( U_{0},\hat{\text{v}}_{j}\right) _{V}\hat{\text{v}}_{j}.$
Therefore, by Bessel's inequality, we find
\begin{equation*}
	\left\Vert U_{n}(0)\right\Vert_{2}\leq \left\Vert U_{0}\right\Vert_{2}\quad \text{and}\quad \left\Vert U_{n}(0)
	\right\Vert _{V}\leq \left\Vert U_{0}\right\Vert _{V}.  
\end{equation*}
Moreover, the dominated convergence theorem yields the following
\begin{equation*}
U_{n}(0) \rightarrow U_{0} \phantom{x} \text{in} \phantom{x} L^{p}(\Omega,V).
\end{equation*}

Our next aim is to show the existence of a global-in-time solution of system \eqref{Galerkin.syst} and to prove the first estimates.

\begin{lem}
Assume that $U_0\in L^{p}(\Omega,V)$ and the hypotheses \eqref{lipschitz1}, \eqref{growth1}, \eqref{lipschitz2}, \eqref{growth2} hold. Then, the system \eqref{Galerkin.syst} admits a unique solution $U_{n}\in L^{2}(\Omega,C([0,T],V_n))$, that verifies the estimates
\begin{align}
&\mathbb{E}\sup_{s\in [0, t]}\|U_{n}(s)\|_{V}^{2}+\dfrac{\beta}{2}\mathbb{E}\int_{0}^{t}\|A_{n}\|_{4}^{4}ds + 2\mu \mathbb{E}\int_{0}^{t}\|\nabla U_{n}\|_{2}^{2}ds\leq C (1 + \mathbb{E}\|U_0\|_{V}^{2}), \qquad \forall t \in [0,T], \label{E.v}
\\ &
\mathbb{E}\|U_{n}\|_{L^{4}(0,t;W^{1,4})}^{4} \leq C (1 + \mathbb{E}\|U_0\|_{V}^{2}), \qquad\forall t \in [0,T], \label{E.14}
\end{align}
where $C$ is a positive constant independent of $n$. \label{estimate3}
\end{lem}
\vspace{0.1cm}
\begin{proof}
Recall \eqref{cutoff}, for fixed $n\in \mathbb{N}$, let us consider the following system
\begin{equation}
\left\{
\begin{array}{l}
d\left(U_n^N -\alpha_{1}P_n\Delta U_{n}^N,\varphi \right)\vspace{2mm}\\
\qquad=\bigl( \mu P_n\Delta U_{n}^N-P_n\phi_N(U_{n}^N)U_{n}^N\cdot \nabla U_n^N + \alpha_1P_n\phi_N(U_{n}^N)\mathrm{div}\left((\nabla U_{n}^N)^{T}A_n^N + A_{n}^N(\nabla U_{n}^N)\right)\vspace{2mm} \\
\qquad +\alpha_{1}P_n\phi_N(U_{n}^N)\mathrm{div}(U_{n}^N\cdot\nabla A_{n}^N)+\alpha_{2}P_n\phi_N(U_{n}^N)\mathrm{div}((A_{n}^N)^{2}) + \beta	P_n\phi_N(U_{n}^N){\rm div}\left(|A_n^N|^2A_n^N\right)
\vspace{2mm} \\
\qquad+P_n\Phi(\cdot,U_{n}^N), \varphi \bigr)dt
+\left( P_n\sigma(\cdot,U_{n}^N),\varphi \right) \,d\mathcal{B}_{t},\qquad  \forall \varphi \in	V_{n},
\vspace{2mm} \\
U_{n}^N(0)=U_{n,0},
\end{array}
\right. \label{Galerkin.syst2}
\end{equation}
Note that \eqref{Galerkin.syst2} defines a globally Lipschitz continuous system of stochastic ordinary differential equations. Hence, by using $e.g.$ the Banach fixed point theorem (\cite[Thm.	1.12]{roubivcek2013nonlinear}), we infer the existence of a unique predictable solution $U_n^N\in	L^2(\Omega,C([0,T],V_n)).$\vspace{0.1cm}

For each $n \in \mathbb{N}$, let us define the sequence $\lbrace\tau_{N}^{n}\rbrace_{N\in \mathbb{N}}$ of stopping times
\begin{equation*}
\tau_{N}^{n} = \inf\lbrace t\geq0, \|U_{n}^N\|_{V} \geq N \rbrace \wedge T, \quad N\in \mathbb{N}.
\end{equation*} 

Recall the equation \eqref{Galerkin.syst2} and take $\varphi :=\text{v}_{j},$ $j=1,...,n$. Set 
\begin{align}
f(U_{n}^N) & :=  \mu \Delta U_{n}^N + \phi_N(U^N_n)[- U_{n}^N\cdot \nabla U_{n}^N+\alpha_{1}\mathrm{div}((\nabla U_{n}^N)^{T}A_{n}^N+A_{n}^N(\nabla U_{n}^N))\notag\\
& + \alpha_{1}\mathrm{div}(U_{n}^N\cdot \nabla A_{n}^N) +\alpha_{2}\mathrm{div}((A_{n}^N)^{2})+\beta\mathrm{div}(\vert A_{n}^N\vert^{2}A_{n}^N)] + \Phi(\cdot,U_{n}^N). \label{fUn}
\end{align}
Then, we get
\begin{equation*}
d(U_{n}^N-\alpha_{1}\Delta U_{n}^N,\text{v}_{j})=(f(U_{n}^N),\text{v}_{j})dt + (\sigma(\cdot,U_{n}^N),\text{v}_{j})d\mathcal{B}_{t},
\end{equation*}
using \eqref{innerV}, it follows that 
\begin{equation*}
d(U_{n}^N,\text{v}_{j})_{V} = (f(U_{n}^N),\text{v}_{j})dt + (\sigma(\cdot,U_{n}^N),\text{v}_{j})d\mathcal{B}_{t}.
\end{equation*}
Applying Itô's formula for the function $x\mapsto x^{2}$  yields
\begin{equation*}
d(U_{n}^N,\text{v}_{j})_{V}^{2} = 2(U_{n}^N,\text{v}_{j})_{V}[(f(U_{n}^N),\text{v}_{j})dt +(\sigma(\cdot,U_{n}^N),\text{v}_{j})d\mathcal{B}_{t}] + \sum_{k=1}^{\infty}\vert(\sigma^{k}(\cdot,U_{n}^N),\text{v}_{j})\vert^{2}dt.
\end{equation*}
Multiplying by $\dfrac{1}{\lambda_{j}}$ and summing over $j=1,...,n,$ give
\begin{equation}
d\|U_{n}^N\|_{V}^{2} = 2(f(U_{n}^N),U_{n}^N)dt + 2(\sigma(\cdot,U_{n}^N),U_{n}^N) d\mathcal{B}_{t} + \sum_{k=1}^{\infty}\sum_{j=1}^{n}\dfrac{1}{\lambda_{j}}\vert(\sigma^{k}(\cdot,U_{n}^N),\text{v}_{j})\vert^{2}dt.\label{ItoV}
\end{equation}

Next, we estimate the $r.h.s$ of the equation \eqref{ItoV}. Denote by $\widetilde{\sigma}_{n}^{k},$ $k \in \mathbb{N}^+,$ the solution to the modified Stokes problem \eqref{Stokes1} for $f = \sigma^{k}(\cdot,U_{n}^N)$, $k \in \mathbb{N}^+.$ Then thanks to \eqref{Stokes1eq}, \eqref{Stokes1ineq} and \eqref{growth2}, we obtain
\begin{align}
\sum_{k=1}^{\infty}\sum_{j=1}^{n}\dfrac{1}{\lambda_{j}}\vert(\sigma^{k}(\cdot,U_{n}^N),\text{v}_{j})\vert^{2} &= \sum_{k=1}^{\infty}\sum_{j=1}^{n}\dfrac{1}{\lambda_{j}}\vert(\widetilde{\sigma}_{n}^{k},\text{v}_{j})_{V}\vert^{2} = \sum_{k=1}^{\infty}\|\widetilde{\sigma}_{n}^{k}\|_{V}^{2} \notag\\ & \leq C\sum_{k=1}^{\infty}\|\sigma^{k}(\cdot,U_{n}^N)\|_{2}^{2}  \leq C\ell(1 +\|U_{n}^N\|_{V}^{2}). \label{Difs}
\end{align}
By the divergence theorem, we have 
\begin{align}
(\mu\Delta U_{n}^N, U_{n}^N)  = -\mu \|\nabla U_{n}^N\|_{2}^{2}.
\end{align}
Due to the anti-symmetry of the trilinear form \eqref{trilin2}, we notice that 
\begin{align}
(U_{n}^N\cdot \nabla U_{n}^N,U_{n}^N)& = b(U_{n}^N,U_{n}^N,U_{n}^N) = -b(U_{n}^N,U_{n}^N,U_{n}^N)= 0.
\end{align}
Using the divergence theorem multiple times together with the symmetry of the matrix $A_n^N$, we obtain
\begin{align}
(\alpha_{1}\text{div}(U_{n}^N\cdot \nabla A_{n}^N),U_{n}^N) = - \alpha_{1}\int_{\mathcal{D}}( U_{n}^N\cdot \nabla A_{n}^N) \cdot \nabla U_{n}^N dx = -\dfrac{\alpha_{1}}{2}\int_{\mathcal{D}} (U_{n}^N\cdot \nabla A_{n}^N)\cdot A_{n}^N dx =0.
\end{align}
Again by the divergence theorem, the symmetry of $A_n^N$ and Young's inequality for  $\varepsilon_{1}, \varepsilon_{2} > 0$, we find 
\begin{align}
\vert(\alpha_{2}\text{div}((A_{n}^N)^{2}), U_{n}^N)\vert & \leq \left\vert \alpha_{2}\int_{\mathcal{D}} (A_{n}^N)^{2}\cdot \nabla U_{n}^N  dx \right\vert  = \left\vert\dfrac{\alpha_{2}}{2}\int_{\mathcal{D}} (A_{n}^N)^{2}\cdot A_{n}^N dx\right\vert\notag\\ & \leq \dfrac{\vert\alpha_{2}\vert}{2}\|A_{n}^{N}\|^{2}_4\|A_{n}^N\|_{2} \leq \dfrac{\varepsilon_{1}}{2}\|A_{n}^N\|_{4}^{4}+\dfrac{\vert\alpha_{2}\vert^{2}}{8\varepsilon_{1}}\|A_{n}^N\|_{2}^{2},\label{ep1}
\end{align} 
\begin{align}
\vert(\alpha_{1}\text{div}((\nabla U_{n}^N)^{T}A_{n}^N+A_{n}^N(\nabla U_{n}^N)),U_{n}^N)\vert & \leq \left\vert\int_{\mathcal{D}} \dfrac{\alpha_1}{2}((\nabla U_{n}^N)^{T}A_{n}^N + A_{n}^N(\nabla U_{n}^N)) \cdot A_{n}^N dx\right\vert\notag\\ & \leq \left\vert \int_{\mathcal{D}}\alpha_1 A_n^N (\nabla U_n^N)\cdot A_n^N dx \right\vert
\leq\alpha_{1}\|A_{n}^N\|_{4}^2\|\nabla U_{n}^N\|_{2} \notag\\& \leq c_{\alpha_{1}}\|A_{n}^N\|_{4}^{2}\|A_{n}^N\|_{2}  \leq \frac{\varepsilon_{2}}{2}\|A_{n}^N\|_{4}^{4} + \dfrac{c_{\alpha_{1}}^{2}}{2\varepsilon_{2}}\|A_{n}^N\|_{2}^{2}.\label{ep2}
\end{align}
The divergence theorem and the symmetry of $A_n^N$, imply
\begin{align}
(\beta\text{div}(\vert A_{n}^N\vert^{2}A_{n}^N),U_{n}^N) &= -\beta\int_{\mathcal{D}} \vert A_{n}^N\vert^{2}A_{n}^N\cdot \nabla U_{n}^N dx = -\dfrac{\beta}{2}\| A_{n}^N\|_{4}^{4}.
\end{align}
Using Young's inequality and \eqref{growth1}, we find
\begin{align}
(\Phi(\cdot,U_{n}^N),U_{n}^N) & \leq \|\Phi(\cdot,U_{n}^N)\|_{2}\|U_{n}^N\|_{2} \leq \dfrac{1}{2}\|\Phi(\cdot,U_{n}^N)\|_{2}^{2} + \dfrac{1}{2}\|U_{n}^N\|_{2}^{2} \leq \dfrac{\ell}{2}(1+\|U_{n}^N\|_{V}^{2}) + \dfrac{1}{2}\|U_{n}^N\|_{V}^{2}.\label{Force}
\end{align}
Taking into account the above inequalities \eqref{Difs}-\eqref{Force}, and choosing $\varepsilon_{1}, \varepsilon_2 = \dfrac{\beta}{4}$, the equation \eqref{ItoV} writes
\begin{align}
d\|U_{n}^N\|_{V}^{2}&+ \dfrac{\beta}{2}\phi_N(U_n^N)\|A_{n}^N\|_{4}^{4}dt+2\mu \|\nabla U_{n}^N\|_{2}^{2}dt \notag\\
&\leq \left( \dfrac{4c_{\alpha_{1}}^{2}+\vert\alpha_{2}\vert^{2}}{\beta} \phi_N(U_n^N)\|A_{n}^N\|_{2}^{2} +C+ C(\ell) \|U_{n}^N\|_{V}^{2}\right)dt + 2(\sigma(\cdot,U_{n}^N),U_{n}^N)d\mathcal{B}_{t} \notag\\ & \leq (C+  C(\alpha_{1},\alpha_{2},\beta,\ell)\|U_{n}^N\|_{V}^{2})dt+ 2(\sigma(\cdot,U_{n}^N),U_{n}^N)d\mathcal{B}_{t}. \label{ineqV1}
\end{align} 
We integrate the inequality \eqref{ineqV1} over the time interval $[0,s]$, for any $s\in [0,\tau_{N}^{n}\wedge t],$ $t\in [0,T]$
\begin{align}
\|U_{n}^N(s)\|_{V}^{2}&+\dfrac{\beta}{2}\int_{0}^{s}\phi_N(U_n^N)\|A_{n}^N\|_{4}^{4}dr + 2\mu \int_{0}^{s}\|\nabla U_{n}^N\|_{2}^{2}dr\notag\\ & \leq \|U_{n}(0)\|_{V}^{2} + C(\ell, T) + C(\alpha_{1},\alpha_{2},\beta,\ell)\int_{0}^{s}\|U_{n}^N\|_{V}^{2}dr + 2\int_{0}^{s}(\sigma(r,U_{n}^N),U_{n}^N)d\mathcal{B}_{r}.\label{INTV}
\end{align} 
The Burkholder-Davis-Gundy inequality together with the estimate \eqref{growth2} and Young's inequality, yield
\begin{align*}
\mathbb{E}&\sup_{s\in [0,\tau_{N}^{n}\wedge t]}\left\vert \int_{0}^{s}(\sigma(r,U_{n}^N),U_{n}^N)d\mathcal{B}_{r}\right\vert \leq C \mathbb{E}\left(\int_{0}^{\tau_{N}^{n}\wedge t}\sum_{k=1}^{\infty}\vert (\sigma^{k}(s,U_{n}^N),U_{n}^N)\vert^{2}ds\right)^{\frac{1}{2}}\notag\\ 
& \leq C \mathbb{E}\left(\int_{0}^{\tau_{N}^{n}\wedge t}\sum_{k=1}^{\infty}\| \sigma^{k}(s,U_{n}^N)\|^{2}_{2}\|U_{n}^N\|^{2}_{2}ds\right)^{\frac{1}{2}} \leq C \mathbb{E}\left(\int_{0}^{\tau_{N}^{n}\wedge t}\ell(1+\|U_{n}^N\|^{2}_{V})\|U_{n}^N\|^{2}_{V} ds\right)^{\frac{1}{2}}\notag\\ 
& \leq  \mathbb{E}\left[(\sup_{s\in [0,\tau_{N}^{n}\wedge t]}\|U_{n}^N\|^{2}_{V})^{\frac{1}{2}}\left(C^{2}\ell\int_{0}^{\tau_{N}^{n}\wedge t}(1+\|U_{n}^N\|^{2}_{V}) ds\right)^{\frac{1}{2}}\right]\notag\\ 
& \leq \dfrac{1}{4}\mathbb{E}\sup_{s\in [0,\tau_{N}^{n}\wedge t]}\|U_{n}^N\|^{2}_{V} + C^{2}\ell \mathbb{E}\int_{0}^{\tau_{N}^{n}\wedge t}(1+\|U_{n}^N\|^{2}_{V}) ds.
\end{align*}
Given the above estimate, we take the supremum on $s\in [0,\tau_{N}^{n}\wedge t]$ and introduce the expectation in the inequality \eqref{INTV}, we obtain 
\begin{align*}
\dfrac{1}{2}\mathbb{E}\sup_{s\in [0,\tau_{N}^{n}\wedge t]}\|U_{n}^N(s)\|_{V}^{2}&+\dfrac{\beta}{2}\mathbb{E}\int_{0}^{\tau_{N}^{n}\wedge t}\phi_N(U_n^N)\|A_{n}^N\|_{4}^{4}dr  + 2\mu \mathbb{E}\int_{0}^{\tau_{N}^{n}\wedge t}\|\nabla U_{n}^N\|_{2}^{2}dr\notag \\ & \leq C + \mathbb{E}\|U_{n}(0)\|_{V}^{2} + C\mathbb{E}\int_{0}^{\tau_{N}^{n}\wedge t}\|U_{n}^N\|_{V}^{2}dr.
\end{align*} 
By Gronwall's inequality, we infer 
\begin{align}
\mathbb{E}\sup_{s\in [0,\tau_{N}^{n}\wedge t]}\|U_{n}^N(s)\|_{V}^{2}&+\dfrac{\beta}{2}\mathbb{E}\int_{0}^{\tau_{N}^{n}\wedge t}\phi_N(U_n^N)\|A_{n}^N\|_{4}^{4}dr + 2\mu \mathbb{E}\int_{0}^{\tau_{N}^{n}\wedge t}\|\nabla U_{n}^N\|_{2}^{2}dr\notag\\ & \leq C (1 + \mathbb{E}\|U(0)\|_{V}^{2}), \qquad \forall t \in [0,T].\label{Gron}
\end{align} 
Therefore, $\mathbb{E}\sup_{s\in [0,\tau_{N}^{n}\wedge T]}\|U_{n}^N(s)\|_{V}^{2} \leq C,$ where $C$ is a positive constant independent of $n$ and $N$. Fix $n \in \mathbb{N}$ and notice that
\begin{align}
\mathbb{E}\sup_{s\in [0,\tau_{N}^{n}\wedge T]}\|U_{n}(s)\|_{V}^{2} & \geq \mathbb{E}(\max_{s\in [0,\tau_{N}^{n}]}\textbf{1}_{\lbrace\tau_{N}^{n}< T\rbrace}\|U_{n}(s)\|_{V}^{2})  \geq N^{2}\mathbb{P}(\tau_{N}^{n}< T).\label{Stopping}
\end{align}
Thus, there	exists a	 subset $\bar{\Omega} \subset \Omega$	with	 full measure \textit{i.e.}	$P(\bar{\Omega})=1,$ and for any $\omega\in	\bar{\Omega}$, there exists $N_0$ such that for all $N\geq N_0$, $\tau_N^n=T,$ see \textit{e.g.} \cite[Theorem 1.2.1]{breckner1999}. Since	$V\hookrightarrow W^*$,	we obtain $\phi_N(U_n^N)=1$	for	all	$s\in[0,T]$	and	$N\geq N_0$. Set $U_n := U_n^{N_0}=\displaystyle\lim_{N\to\infty}U_n^N$ with respect to the	$V$-norm, hence the system \eqref{Galerkin.syst2} becomes of the following form $\mathbb{P}-a.s.$ in $\Omega$
\begin{equation}
\left\{
\begin{array}{l}
d\left( U_n -\alpha_{1}P_n\Delta U_{n},\varphi \right) =\bigl( \mu P_n\Delta U_{n}-P_nU_{n}\cdot \nabla U_n + \alpha_1P_n\mathrm{div}\left((\nabla U_{n})^{T}A_n + A_{n}(\nabla U_{n})\right)\vspace{2mm} \\
\vspace{2mm} 
\qquad\qquad\qquad\qquad +\alpha_{1}P_n\mathrm{div}(U_{n}\cdot\nabla A_{n})+\alpha_{2}P_n\mathrm{div}(A_{n}^{2}) + \beta	P_n{\rm div}\left(|A_n|^2A_n\right)
\vspace{2mm} \\
\vspace{2mm} 		
\qquad\qquad\qquad\qquad+P_n\Phi(t,U_{n}), \varphi \bigr)dt
+\left(P_n\sigma(t,U_{n}),\varphi \right) \,d\mathcal{B}(t),\quad  \forall \varphi \in V_{n}, \phantom{x} t\in[0,T],
\vspace{2mm} \\
U_{n}(0)=U_{n,0}.
\end{array}
\right. \label{Galerkin.syst3}
\end{equation}
Moreover, since $\tau_{N}^{n} \to T$ in probability, as $N \to \infty,$ and the sequence of stopping times $\lbrace \tau_{N}^{n}\rbrace_{N}$ is monotone for any fixed $n$, one can apply the monotone convergence theorem in order to pass to the limit in the inequality \eqref{Gron}, as $N \to \infty$
\begin{align*}
\mathbb{E}\sup_{s\in [0, t]}\|U_{n}(s)\|_{V}^{2}&+\dfrac{\beta}{2}\mathbb{E}\int_{0}^{t}\|A_{n}\|_{4}^{4}ds + 2\mu \mathbb{E}\int_{0}^{t}\|\nabla U_{n}\|_{2}^{2}ds \leq C (1 + \mathbb{E}\|U(0)\|_{V}^{2}), \qquad \forall t \in [0,T].
\end{align*} 
From the above inequality, we notice that
\begin{equation*}
\mathbb{E}\int_{0}^{t}\|A_{n}\|_{4}^{4}ds \leq C(1+\mathbb{E}\|U_{0}\|_{V}^{2}), \phantom{xxxxxxx} \forall t \in [0,T].
\end{equation*}
Then using \eqref{KORN*}, implies 
\begin{equation*}
\mathbb{E}\|U_{n}\|_{L^{4}(0,t;W^{1,4})}^{4}\leq C(1+\mathbb{E}\|U_{0}\|_{V}^{2}), \phantom{xxxxx} \forall t\in [0,T].
\end{equation*}
\end{proof} 

\subsection{Uniform estimates}

\begin{lem}
Assume that $U_0\in L^{p} (\Omega,V)$ and the hypotheses \eqref{lipschitz1}, \eqref{growth1}, \eqref{lipschitz2}, \eqref{growth2} hold. Then, the solution $U_{n}$ to the equation \eqref{Galerkin.syst3} belongs to $L^{p}(\Omega,C([0,T],V_{n}))$ and verifies 
\begin{equation}
\mathbb{E}\sup_{s\in [0, t]}\|U_{n}(s)\|_{V}^{p} \leq C (1 + \mathbb{E}\|U_0\|_{V}^{p}), \label{E.1}
\end{equation}
for all $t \in [0,T]$, where $C$ is a positive constant independent of $n$. \label{estimate1}
\end{lem}
\begin{proof}
For each $n \in \mathbb{N}$, consider the following sequence of stopping times
\begin{equation*}
\tau_{N}^{n}= \inf\lbrace t \geq 0, \|U_{n}\|_{V}\geq N\rbrace \wedge T, \quad N\in \mathbb{N}.
\end{equation*}

Recall the equation \eqref{Galerkin.syst3} and take $\varphi:=\text{v}_j,$ $j=1,\cdots,n.$ Set 
\begin{align}
f(U_{n})  :=  \mu \Delta U_{n} - U_{n}\cdot \nabla U_{n} &+\alpha_{1}\mathrm{div}(U_{n}\cdot \nabla A_{n}+(\nabla U_{n})^{T}A_{n}+A_{n}(\nabla U_{n}))\notag\\
& +\alpha_{2}\mathrm{div}(A_{n}^{2})+\beta\mathrm{div}(\vert A_{n}\vert^{2}A_{n}) + \Phi(\cdot,U_{n}).\label{f(Un)2}
\end{align}
Then, using \eqref{innerV}, we obtain
\begin{equation*}
d(U_{n},\text{v}_{j})_{V} = (f(U_{n}),\text{v}_{j})dt + (\sigma(\cdot,U_{n}),\text{v}_{j})d\mathcal{B}_{t}.
\end{equation*}
Applying Itô's formula for $x\mapsto x^2$, then multiplying by $\dfrac{1}{\lambda_{j}}$ and summing over $j=1,...,n,$ yield
\begin{equation*}
d\|U_{n}\|_{V}^{2} = 2(f(U_{n}),U_{n})dt + 2(\sigma(\cdot,U_{n}),U_{n}) d\mathcal{B}_{t} + \sum_{k=1}^{\infty}\sum_{j=1}^{n}\dfrac{1}{\lambda_{j}}\vert(\sigma^{k}(\cdot,U_{n}),\text{v}_{j})\vert^{2}dt.
\end{equation*}
Taking into account the inequalities \eqref{Difs}-\eqref{Force} with $U_n:=\displaystyle\lim_{N\to\infty}U_n^N,$ (see the details in the proof of Lemma \ref{estimate3}), and choose $\varepsilon_{1}, \varepsilon_2 =\frac{\beta}{2}$ in \eqref{ep1}, \eqref{ep2}, we obtain
\begin{align*}
d\|U_{n}\|_{V}^{2} \leq \left(\dfrac{4c_{\alpha_{1}}^{2}+\vert\alpha_{2}\vert^{2}}{2\beta}\|A_{n}\|_{2}^{2}+ C(1+\|U_{n}\|_{V}^{2})\right)dt + 2(\sigma(\cdot,U_{n}),U_{n})d\mathcal{B}_{t}.
\end{align*}
Integrating the above inequality over the time interval $[0,s]$, $0\leq s \leq \tau_{N}^{n}\wedge t,$ for $t \in [0,T]$, and taking both sides to the power $q$, for $q>1$, we find 
\begin{align}
\|U_{n}(s)\|_{V}^{2q}&\leq C_{q}\left(\|U_{n}(0)\|_{V}^{2q} + \int_{0}^{s}\left(C(\alpha_{1},\alpha_{2},\beta,q)\|U_{n}\|_{V}^{2q}+ C(1+\|U_{n}\|_{V}^{2q})\right)dr+ 2\left\vert\int_{0}^{s}(\sigma(r,U_{n}),U_{n})d\mathcal{B}_{r}\right\vert^q\right)\notag\\
& \leq C_{q}\left(\|U_{n}(0)\|_{V}^{2q} + C\int_{0}^{s} (1+\|U_{n}\|_{V}^{2q})dr+ 2\left\vert\int_{0}^{s}(\sigma(r,U_{n}),U_{n})d\mathcal{B}_{r}\right\vert^q\right).\label{INEQVq}
\end{align}
Using the Burkholder-Davis-Gundy inequality, the assumption \eqref{growth2} and Young's inequality, we get
\begin{align}
\mathbb{E}&\sup_{s\in [0,\tau_{N}^{n}\wedge t]}\left\vert \int_{0}^{s}(\sigma(r,U_{n}),U_{n})d\mathcal{B}_{r}\right\vert^q  \leq C \mathbb{E}\left( \int_{0}^{\tau_{N}^{n}\wedge t}\sum_{k=1}^{\infty}\vert (\sigma^{k}(r,U_{n}),U_{n})\vert^{2}ds\right)^{\frac{q}{2}}\notag\\
& \leq C \mathbb{E}\left( \int_{0}^{\tau_{N}^{n}\wedge t} \sum_{k=1}^{\infty}\|\sigma^{k}(r,U_{n})\|_{2}^{2}\|U_{n}\|^{2}_{2} ds\right)^{\frac{q}{2}}
\leq C\mathbb{E}\left( \int_{0}^{\tau_{N}^{n}\wedge t} \ell(1+\|U_{n}\|_{V}^{2})\|U_{n}\|_{V}^{2}ds\right)^{\frac{q}{2}} \notag\\ & \leq C\mathbb{E}\left( \int_{0}^{\tau_{N}^{n}\wedge t}\ell^q(1+\|U_{n}\|_{V}^{2q})\|U_{n}\|_{V}^{2q}ds\right)^{\frac{1}{2}} \leq \mathbb{E}\left[(\sup_{s\in [0,\tau_{N}^{n}\wedge t]}\|U_{n}\|_{V}^{2q})^{\frac{1}{2}}\left(C^{2}\ell^q\int_{0}^{\tau_{N}^{n}\wedge t}(1+\|U_{n}\|_{V}^{2q})ds\right)^{\frac{1}{2}}\right]\notag\\ &
\leq \dfrac{1}{4C_{q}}\mathbb{E}(\sup_{s\in [0,\tau_{N}^{n}\wedge t]}\|U_{n}\|_{V}^{2q})+ C^{2}(q)\ell^q \mathbb{E}\int_{0}^{\tau_{N}^{n}\wedge t}(1+\|U_{n}\|_{V}^{2q})ds.\label{Burkh}
\end{align}
Taking the supremum on $s \in [0,\tau_{N}^{n}\wedge t]$, applying the expectation and then substituting \eqref{Burkh} in the inequality \eqref{INEQVq}, we derive
\begin{align*}
\dfrac{1}{2}\mathbb{E}\sup_{s\in [0,\tau_{N}^{n}\wedge t]}\|U_{n}(s)\|_{V}^{2q} &\leq C \mathbb{E}\|U_{n}(0)\|_{V}^{2q} + C +C\mathbb{E}\int_{0}^{\tau_{N}^{n}\wedge t} \|U_{n}\|_{V}^{2q}dr.
\end{align*}
By Gronwall's inequality, we deduce that for any $q > 1$ and $t \in [0,T],$
\begin{equation}
\mathbb{E}\sup_{s\in[0,\tau_{N}^{n}\wedge t]} \|U_{n}(s)\|_{V}^{2q} \leq C(1+\mathbb{E}\|U_{0}\|_{V}^{2q}).\label{Gron2}
\end{equation}
Therefore, $\mathbb{E}\sup_{s\in[0,\tau_{N}^{n}\wedge T]} \|U_{n}(s)\|_{V}^{2q} \leq C,$ where $C$ is positive and independent of $n$ and $N$. Now, fix $n \in \mathbb{N}$, and recall the definition of $\tau_{N}^{n}$, we write
\begin{align*}
\mathbb{E}\sup_{s \in [0,\tau_{N}^{n}\wedge T]}\|U_{n}(s)\|_{V}^{2q} &\geq \mathbb{E}(\max_{s\in [0,\tau_{N}^{n}]}\textbf{1}_{\lbrace \tau_{N}^{n}< T\rbrace} \|U_{n}(s)\|_{V}^{2q}) \geq N^{2}\mathbb{P}(\tau_{N}^{n}< T).
\end{align*}
As in the proof of Lemma \ref{estimate3}, notice that $\tau_{N}^{n} \to T$ in probability, as $N\to \infty$. Therefore, we deduce the existence of a subsequence $\lbrace\tau_{N_{k}}^{n}\rbrace \subset \lbrace\tau_{N}^{n}\rbrace$, that may depend on $n$, such that 
$\tau_{N_{k}}^{n}\to T$ \text{a.e.}, \text{as} $k\to \infty.$ Moreover, the sequence of stopping times $\lbrace \tau_{N}^{n}\rbrace_{N \in \mathbb{N}}$ is monotone for any fixed $n$, then by the monotone convergence theorem we can pass to the limit in the inequality \eqref{Gron2}, as $N \to \infty$,
\begin{equation*}
\mathbb{E}\sup_{s\in[0, t]} \|U_{n}(s)\|_{V}^{2q} \leq C(1+\mathbb{E}\|U_{0}\|_{V}^{2q}),
\end{equation*}
for any $q > 1$ and $t \in [0,T]$. Taking $q = \dfrac{p}{2}$ concludes the proof.
\end{proof}

\begin{lem}
Assume that $U_0 \in L^p(\Omega,V)$ and the hypotheses \eqref{lipschitz1}, \eqref{growth1}, \eqref{lipschitz2}, \eqref{growth2} hold. Then, the solution $U_n$ to the system \eqref{Galerkin.syst3} verifies for some $\delta \in (0,1)$
\begin{equation}
\mathbb{E}\,\|U_{n}\|_{\mathcal{C}^{0,\delta}([0,T],H)}\leq C,
\label{estimate2}
\end{equation}
where $C$ is a positive constant independent of $n$.
\label{lemma.4}
\end{lem}

\begin{proof}
The estimate \eqref{E.1} yields 
\begin{equation}
\mathbb{E}\,\sup_{t\in[0,T]}\|U_{n}(t)\|^p_2\leq C. \label{NormC}
\end{equation}
Next, we show that for a suitable $\delta \in (0,1),$ we have the following
\begin{equation*}
\mathbb{E} \sup_{s,t \in [0,T], s\neq t}\dfrac{\|U_n(t) -U_{n}(s)\|_{2}}{\vert t-s\vert^{\delta}}\leq C.
\end{equation*}
Recall \eqref{f(Un)2}, then setting $\varphi := \text{v}_j,$ $j=1,\cdots,n,$ the equation in \eqref{Galerkin.syst3} becomes as follows
\begin{equation*}
d(U_{n},\text{v}_j)_{V}=(\tilde{f}_{n},\text{v}_j)_{V}dt + \sum_{k=1}^{\infty}(\tilde{\sigma}_{n}^k,\text{v}_j)_{V}d\beta_k(t), \qquad \forall j=1,\cdots,n,
\end{equation*}
where $\tilde{f}_{n}$ and $\tilde{\sigma}_{n}^k,$ $k\in \mathbb{N}^+$, stand for the solutions to the modified Stockes problem \eqref{Stokes1} for $f= P_n f(U_{n})$ and $f = P_n\sigma^k(\cdot,U_{n})$, $k\in \mathbb{N}^+$, respectively. Thus, the following equation holds in $V_n^{*}$
\begin{equation*}
U_{n}(t)-U_{n,0} = \int_{0}^{t}\tilde{f}_{n}ds + \sum_{k=1}^{\infty}\int_{0}^{t}\tilde{\sigma}_{n}^k d\beta_{k}(s),
\end{equation*}
moreover, we get for $0\leq s \leq t\leq T$
\begin{equation*}
(U_{n}(t)-U_{n}(s),\mathrm{v}_{j})_{V}=\int_{s}^{t}(\tilde{f}_{n},\mathrm{v}_{j})_{V}dr +\sum_{k=1}^{\infty}\int_{s}^{t}(\tilde{\sigma}_{n}^k,\mathrm{v}_{j})_{V}d\beta_{k}(r), \qquad \forall j=1,...,n.
\end{equation*}
Recall \eqref{basis}, then multiplying both sides of the above equation by $\frac{1}{\lambda_{j}}$, yields
\begin{equation*}
(U_{n}(t)-U_{n}(s),\mathrm{v}_{j})=\int_{s}^{t}(\tilde{f}_{n},\mathrm{v}_{j})dr + \sum_{k=1}^{\infty}\int_{s}^{t}(\tilde{\sigma}_{n}^k,\mathrm{v}_{j})d\beta_{k}(r), \qquad \forall j=1,...,n.
\end{equation*}
Take $\varphi \in V_{n}$, then we write $\varphi = \sum_{j=1}^{n}\xi_{j}\mathrm{v}_{j}$, for all $\xi_{j} \in \mathbb{R}$, consequentely we obtain
\begin{equation}
(U_{n}(t)-U_{n}(s),\varphi)=\int_{s}^{t}(\tilde{f}_{n},\varphi)dr + \sum_{k=1}^{\infty}\int_{s}^{t}(\tilde{\sigma}_{n}^k,\varphi)d\beta_{k}(r).\label{EQinVn*}
\end{equation}
Now, let $\psi \in H$ then by Remark \ref{lem-proj-H}, $P_n\psi \in	V_n$. Moreover, the following equality	holds
\begin{align*}
(U_{n}(t)-U_{n}(s),\psi)=(P_nU_{n}(t)-P_nU_{n}(s),\psi)=(U_{n}(t)-U_{n}(s),P_n^*\psi)=(U_{n}(t)-U_{n}(s),P_n\psi),
\end{align*}
where we used the fact that $P_n$ is	a self-adjoint operator	$i.e.$ $P_n=P_n^*$. Then, from \eqref{EQinVn*}, it follows
\begin{equation*}
(U_{n}(t)-U_{n}(s),P_n\psi)=\int_{s}^{t}(\tilde{f}_{n},P_n\psi)dr +\sum_{k =1}^\infty\int_{s}^{t}(\tilde{\sigma}_{n}^k,P_n\psi)d\mathcal{\beta}_{k}(r),\quad	\forall	\psi\in	H.
\end{equation*}
Thus
\begin{equation*}
(U_{n}(t)-U_{n}(s),\psi)=\int_{s}^{t}(P_n\tilde{f}_{n},\psi)dr +\sum_{k =1}^\infty\int_{s}^{t}(P_n\tilde{\sigma}_{n}^k,\psi)d\mathcal{\beta}_{k}(r),\quad	\forall	\psi\in	H,
\end{equation*}
and	the	following equation holds	in the $H$-sense	
\begin{equation*}
U_{n}(t)-U_{n}(s)=\int_{s}^{t}P_n\tilde{f}_{n}dr +\sum_{k =1}^\infty\int_{s}^{t}P_n\tilde{\sigma}_{n}^kd\mathcal{\beta}_{k}(r),\quad	\forall	t \in [0,T].
\end{equation*}

Since $H \simeq H^{*}$, we have $\|u\|_{2} = \sup_{\|\psi\|_{2}\leq 1}\vert (u,\psi)\vert,$ for all $u \in H.$ Thus, for any $\psi \in H$
\begin{align}
\|&U_{n}(t)-U_{n}(s)\|_{2} = \sup_{\|\psi\|_{2}\leq 1} \vert(U_{n}(t)-U_{n}(s),\psi)\vert = \sup_{\|\psi\|_{2}\leq 1}\left\vert\int_{s}^{t}(P_n\tilde{f}_{n},\psi)dr + \sum_{k=1}^{\infty}\int_{s}^{t}(P_n\tilde{\sigma}_{n}^k,\psi)d\beta_{k}(r)\right\vert\notag\\
&\leq \int_{s}^{t}\|\tilde{f}_{n}\|_{2}\|P_n\|_{L^2(H,V_n)}dr + \left\|\sum_{k=1}^{\infty}\int_{s}^{t}P_n\tilde{\sigma}_{n}^k d\beta_{k}(r)\right\|_{2} 
\leq \underbrace{\int_{s}^{t}\|\tilde{f}_{n}\|_{2}dr}_{\|\int_s^t\text{det}(U_{n})\|_{2}} + \underbrace{\left\|\sum_{k=1}^{\infty}\int_{s}^{t}P_n\tilde{\sigma}_{n}^k d\beta_{k}(r)\right\|_{2}}_{\|\int_{s}^t\text{stoch}(U_{n})\|_{2}}, \label{detstoch} 
\end{align}
where we used $\|P_n\|_{L^2(H,V_n)} \leq 1$ in the last inequality, see Remark \ref{lem-proj-H}. Thanks to \eqref{Stokes2ineq}, we have
\begin{align}
\int_s^t\|\widetilde{f}_n\|_{2}dr \leq C \int_{s}^{t}\|f(U_{n})\|_{(V \cap H^{2})^{*}}dr.\label{ineq2}
\end{align}
By definition, we know
\begin{equation*}
\|f(U_{n})\|_{(V \cap H^{2})^{*}} = \sup_{\|\varphi\|_{V\cap H^{2}}\leq 1}\vert(f(U_{n}),\varphi)\vert.
\end{equation*}
Next, we estimate each term of the $r.h.s$ of above equality. Using the divergence theorem and the embeddings $V \hookrightarrow \mathbf{L}^{2}(\mathcal{D}),$ $W^{1,4}(\mathcal{D}) \hookrightarrow L^{4}(\mathcal{D})$, $H^{2}(\mathcal{D})\hookrightarrow W^{1,4}(\mathcal{D})$, we obtain
\begin{align}
\vert (\Delta U_{n},\varphi)\vert = \vert(\nabla U_{n}, \nabla\varphi)\vert \leq C\| U_{n}\|_{V}\|\varphi\|_{V},\label{1}
\end{align}
\begin{align}
 \vert(U_{n}.\nabla U_{n},\varphi)\vert  &\leq \|U_{n}\|_{4}\|\nabla U_{n}\|_{4}\|\varphi\|_{2}\leq C\|U_{n}\|_{W^{1,4}}^{2}\|\varphi\|_{V},\label{2}
\end{align}
\begin{align}
\vert (\mathrm{div}((\nabla U_{n})^{T}A_{n}+A_{n}(\nabla U_{n})),\varphi)\vert
&\leq \vert((\nabla U_{n})^{T}A_{n},\nabla \varphi)\vert + \vert(A_{n}(\nabla U_{n}),\nabla \varphi)\vert  \notag\\
& \leq C\|\nabla U_{n}\|_{4}\|\nabla U_{n}\|_{4}\|\nabla \varphi\|_{2}\leq C \|U_{n}\|_{W^{1,4}}^{2}\|\varphi\|_{V},\label{3}
\end{align}
\begin{align}
\vert (\mathrm{div}(U_{n}\cdot \nabla A_{n}),\varphi)\vert &\leq  \vert (U_{n}.\nabla A_{n},\nabla \varphi)\vert\leq \left\vert\int_{\mathcal{D}}\sum_{i,j=1}^{d}U_{n}^{i}A_{n}^{j}\cdot \partial_{i}\partial_{j}\varphi dx\right\vert \notag\\
&\leq \|U_{n}\|_{4}\|\nabla U_{n}\|_{4}\|\varphi\|_{H^2} \leq C\|U_{n}\|_{W^{1,4}}^{2}\|\varphi\|_{H^{2}},\label{4}
\end{align}
\begin{align}
\vert \alpha_{2}(\mathrm{div}(A_{n}^{2}),\varphi)\vert &\leq \vert \alpha_{2}\vert \vert(A_{n}^{2},\nabla \varphi)\vert\leq C \|\nabla U_{n}\|_{4}^{2}\|\nabla \varphi\|_{2} \leq C \|U_{n}\|_{W^{1,4}}^{2}\|\varphi\|_{V},\label{5}
\end{align}
\begin{align}
\vert(\mathrm{div}(\vert A_{n}\vert^{2}A_{n}),\varphi)\vert &\leq  \vert(\vert A_{n}\vert^{2}A_{n},\nabla \varphi)\vert
\leq C\|\nabla U_{n}\|_{4}^{3}\|\nabla \varphi\|_{4}
\notag\\
&\leq C\|U_{n}\|_{W^{1,4}}^{3}\|\varphi\|_{W^{1,4}} \leq C \|U_{n}\|_{W^{1,4}}^{3}\|\varphi\|_{H^{2}}.\label{6}
\end{align}
Moreover by \eqref{growth1}, we have
\begin{align}
 \vert(\Phi(\cdot,U_{n}),\varphi)\vert &\leq \|\Phi(\cdot,U_{n})\|_{2}\|\varphi\|_{2}\leq \ell (1+\|U_{n}\|_{V}^{2})\|\varphi\|_{V}\leq \ell \|\varphi\|_{V}+\ell\|U_{n}\|_{V}^{2}\|\varphi\|_{V}.\label{7}
\end{align}
Combining the estimates \eqref{1}-\eqref{7}, we obtain
\begin{align*}
\vert (f(U_{n}),\varphi)\vert & \leq C (1+\|U_{n}\|_{V}+\|U_{n}\|_{V}^{2}+\|U_{n}\|_{W^{1,4}}^{2}+\|U_{n}\|_{W^{1,4}}^{3})\|\varphi\|_{V\cap H^{2}},
\end{align*}
thus 
\begin{equation}
\|f(U_{n})\|_{(V\cap H^{2})^{*}}\leq C(1+\|U_{n}\|_{V}+\|U_{n}\|_{V}^{2}+\|U_{n}\|_{W^{1,4}}^{2}+\|U_{n}\|_{W^{1,4}}^{3}).\label{VHf}
\end{equation}
Integrating the inequality \eqref{VHf} over $[s,t],$ $t \in [0,T]$, then applying Hölder's inequality, we derive 
\begin{align*}
\int_{s}^{t}\|f(U_{n})\|_{(V\cap H^{2})^{*}}dr \leq C (t-s)^{\frac{1}{4}}\left(\int_{s}^{t} (1+\|U_{n}\|_{V}^{\frac{8}{3}} +\|U_{n}\|_{W^{1,4}}^{4})dr\right)^{\frac{3}{4}}.
\end{align*}
Take the supremum over the time interval $[0,T]$ and apply the expectation to both sides of the above inequality, we obtain
\begin{align*}
\mathbb{E}\sup_{s,t\in[0,T]}\int_{s}^{t}\|f(U_{n})\|_{(V\cap H^{2})^{*}}dr \leq C (t-s)^{\frac{1}{4}}\left(\mathbb{E}\int_{0}^{T}(1+\|U_{n}\|_{V}^{\frac{8}{3}} +\|U_{n}\|_{W^{1,4}}^{4})dt\right)^{\frac{3}{4}}.
\end{align*}
Therefore, the estimates \eqref{E.1} and \eqref{E.14} imply
\begin{equation*}
\mathbb{E}\sup_{s,t\in[0,T], s\neq t}\dfrac{\int_{s}^{t}\|f(U_{n})\|_{(V\cap H^{2})^{*}}dr}{(t-s)^{\frac{1}{4}}}\leq C.
\end{equation*}
It follows then for any $\delta_{1}\in (0,\frac{1}{4}]$
\begin{equation}
\mathbb{E}\|\text{det}(U_{n})\|_{\mathcal{C}^{0,\delta_{1}}([0,T],H)}\leq C. \label{DetU}
\end{equation}

Following \cite[Lemma 2.1]{flandoli1995martingale}, the stochastic term can be estimated in a suitable fractional Sobolev space. Namely, using the Burkholder-Davis-Gundy inequality, we obtain for any $s \in [0,t]$
\begin{align*}
\mathbb{E}\left\|\sum_{k=1}^{\infty}\int_{s}^{t}P_n\tilde{\sigma}_{n}^k d\beta_{k}(r)\right\|_{2}^{p} &\leq \mathbb{E}\sup_{\tau \in [s,t]}\left\|\sum_{k=1}^{\infty}\int_{s}^{\tau}P_n\tilde{\sigma}_{n}^k d\beta_{k}(r)\right\|_{2}^{p} \leq C \mathbb{E}\left(\int_{s}^{t}\sum_{k=1}^{\infty}\|P_n\tilde{\sigma}_{n}^{k}\|_{2}^{2}dr\right)^{\frac{p}{2}} \notag\\ &
\leq C \mathbb{E}\left(\int_{s}^{t}\sum_{k=1}^{\infty}\|\tilde{\sigma}_{n}^{k}\|_{2}^{2}dr\right)^{\frac{p}{2}}
\leq C(t-s)^{\frac{p}{2}-1}\mathbb{E}\int_{s}^{t}\left(\sum_{k=1}^{\infty}\|\tilde{\sigma}_{n}^{k}\|_{2}^{2}\right)^{\frac{p}{2}}dr.
\end{align*}
Therefore, using \eqref{Stokes1ineq} and \eqref{growth2}, we obtain for any $ s \in [0,t] $
\begin{align*}
\mathbb{E}\left\|\sum_{k=1}^{\infty}\int_{s}^{t}P_n\tilde{\sigma}_{n}^k d\beta_{k}(r)\right\|_{2}^{p} &\leq  C(t-s)^{\frac{p}{2}-1}\mathbb{E}\int_{s}^{t}\left(\sum_{k=1}^{\infty}\|\sigma^{k}(r,U_{n})\|_{2}^{2}\right)^{\frac{p}{2}}dr \\ \notag 
& \leq C(t-s)^{\frac{p}{2}-1}\mathbb{E}\int_{s}^{t}(1+\|U_{n}\|_{V}^{2})^{\frac{p}{2}}dr
\leq C(t-s)^{\frac{p}{2}}(1+\mathbb{E}\sup_{r\in [s,t]}\|U_{n}\|_{V}^{p}).
\end{align*}
Integrating twice over the time interval $[0,T]$, we obtain for $\gamma \in (0,\frac{1}{2})$
\begin{align*}
\mathbb{E}\int_{0}^{T}\int_{0}^{T}\dfrac{\|\sum_{k=1}^\infty\int_{s}^{t}P_n\tilde{\sigma}_{n}^k d\beta_{k}(r)\|_{2}^{p}}{\vert t-s\vert^{1+\gamma p}}dsdt &\leq C (1+\mathbb{E}\sup_{t\in [0,T]}\|U_{n}\|_{V}^{p})\int_{0}^{T}\int_{0}^{T}\vert t-s\vert^{\frac{p}{2}-1-\gamma p}dsdt \leq C.
\end{align*}
Notice that it is straightforward to check that 
\begin{equation*}
\int_{0}^{T}\int_{0}^{T}\vert t-s\vert^{\frac{p}{2}-1-\gamma p}dsdt = \dfrac{2}{\frac{p}{2}-\gamma p}\int_{0}^{T}t^{\frac{p}{2}-\gamma p}dt < \infty.
\end{equation*}
Thus, for $\gamma\in(0,\frac{1}{2})$
\begin{equation*}
\mathbb{E}\|\text{stoch}(U_{n})\|^{p}_{W^{\gamma,p}(0,T; H)}\leq C. 
\end{equation*}
Given $p> 2$, recall from $e.g.$ \cite[Section 2] {flandoli1995martingale}, that
$W^{\gamma,p}(0,T;H) \hookrightarrow \mathcal{C}^{0,\delta_{2}}([0,T],H)$ \text{for} $\delta_{2} \in (0,\gamma p-1).$ Then, for any $\gamma \in (0,\frac{1}{2})$ such that $\gamma p>1$, we have for $\delta_{2}\in (0,\gamma p-1)$
\begin{equation}
\mathbb{E}\|\text{stoch}(U_{n})\|^{p}_{\mathcal{C}^{0,\delta_{2}}([0,T], H)}\leq C. \label{StochU}
\end{equation}
Combining the estimates \eqref{NormC}, \eqref{DetU} and \eqref{StochU}, we deduce that there exists $\delta= \min(\delta_{1},\delta_{2})$ such that
\begin{equation}
\mathbb{E}\|U_{n}\|_{\mathcal{C}^{0,\delta}([0,T], H)}\leq C.\label{HolderU}
\end{equation}
Therefore, $U_{n}$ is bounded in $L^{1}(\Omega, \mathcal{C}^{0,\delta}([0,T],H))$.
\end{proof}

\subsection{Compactness criteria}
Define the stochastic process $\Upsilon_{n}(t) := (\mathcal{B}(t), U_{n}(t),U_{n}(0))$
and the space $\Xi := C([0,T],H)\times V.$ The trajectories of $\Upsilon_{n}(t)$ belong to the product space $C([0,T],\mathcal{U}_{1})\times \Xi.$

Let $\mathcal{L}_{n}:=\mathcal{L}(\Upsilon_{n}),$ $n \in \mathbb{N}$, given as
$$\mathcal{L}_{n}(\mathcal{A}) = \mathbb{P}(\Upsilon_{n} \in \mathcal{A}), \phantom{x} \text{for any Borel set} \phantom{x} \mathcal{A} \in \mathfrak{B}(C([0,T],\mathcal{U}_{1})\times \Xi),$$
to be the law of the random variable $\Upsilon_{n}: \Omega \to C([0,T],\mathcal{U}_{1}) \times \Xi.$ Next, we show the tightness of the family of laws $\lbrace\mathcal{L}_n\rbrace_{n\in \mathbb{N}}.$ In particular, for $n \in \mathbb{N}$, denote by $\mathcal{L}_{\mathcal{B}},$ $\mathcal{L}_{U_{n}}$ and $\mathcal{L}_{U_{n,0}},$ the laws of $\mathcal{B}$, $U_{n}$ and $U_{n}(0)$ on $C([0,T], \mathcal{U}_{1})$, $C([0,T],H)$ and $V$ respectively.\vspace{0.2cm}

Let $\theta > 0,$ since $\mathcal{L}_{\mathcal{B}}:= \mathcal{L}(\mathcal{B})$ is a probability measure on $C([0,T],\mathcal{U}_{1})$, then it defines a Radon measure on $C([0,T], \mathcal{U}_{1})$. Hence, there exists a compact set $S_{L_\theta} \subset C([0,T], \mathcal{U}_{1})$ such that $\mathcal{L}_{\mathcal{B}}(S_{L_\theta})\geq 1-\frac{\theta}{3}.$\vspace{0.1cm}

For any $M >0$, consider the set 
$$S_{M}:= \left\{u\in L^p(0,T;V)\cap \mathcal{C}^{0,\delta}([0,T],H): \|u\|_{\delta,p, V, H}\leq M\right\},$$ 
where $\|\cdot\|_{\delta,p,V,H}$ is defined by \eqref{norm.com}, for $\delta$ given as in \eqref{HolderU}. Lemma \ref{C.2} ensures the compact embedding 
$L^{p}(0,T;V) \cap \mathcal{C}^{0,\delta}([0,T],H) \hookrightarrow C([0,T],H).$
Therefore, the set $S_M$ is a relatively compact subset of $C([0,T],H).$ Moreover, by \eqref{E.1} and \eqref{estimate2} we have
$$\mathbb{P}(U_{n}\notin S_M)=\mathbb{P}(\|U_{n}\|_{\delta,p,V,H}> M)\leq \frac{1}{M^2}\mathbb{E}\|U_{n}\|_{\delta,p,V,H}^2\leq \frac{C}{M^2}.$$
Therefore, we choose $M_\theta$ large enough such that $\mathbb{P}(U_{n}\notin S_{M_\theta})\leq \frac{C}{M_\theta^2}<\frac{\theta}{3}.$
Thus, $\mathcal{L}_{U_{n}}(S_{M_{\theta}}) \geq 1-\frac{\theta}{3},$ $\forall n \in \mathbb{N}$.

Since $U_{n}(0)$ converges strongly to $U_{0}$ in $L^{p}(\Omega,V)$, then $U_{n}(0)$ converges to $U_{0}$ in law in $V$ and by Prokhorov's theorem there exists a compact set $S_{N_\theta} \subset V$, such that $\mathcal{L}_{U_{n,0}}(S_{N_\theta}) \geq 1-\frac{\theta}{3}$, $\forall n \in \mathbb{N}.$

Consequently, we deduce the existence of a relatively compact set
$$S_{L_\theta}\times S_{M_\theta} \times S_{N_\theta}
\in \mathfrak{B}(C([0,T],\mathcal{U}_{1})\times \Xi)$$
such that $\mathcal{L}_{n}(S_{N_\theta}\times S_{M_\theta}\times S_{\theta})\geq 1-\theta$, $\forall n\in \mathbb{N}$, and thus, the tightness of the family $\{\mathcal{L}_{n}\}_{n\in \mathbb{N}}$ on the space $C([0,T],\mathcal{U}_{1})\times \Xi$ holds. Hence, by Prokhorov's Theorem, we deduce the existence of a subsequence  of $\{\mathcal{L}_{n}\}_{n\in \mathbb{N}},$ still denoted $\{\mathcal{L}_{n}\}_{n \in \mathbb{N}},$ that converges weakly to a probability measure $\mathcal{L}$ on $C([0,T],\mathcal{U}_{1})\times \Xi$.

\subsection{Convergence results}
\subsubsection{Step 1} 
Since $\mathcal{L}_{n} \rightarrow \mathcal{L}$ weakly as $n \to \infty$, then by the modified version of Skorokhod's Theorem \cite[Theorem C.1]{brzezniak2018stochastic}, there exist a probability space $(\widetilde{\Omega}, \widetilde{\mathcal{F}},\widetilde{\mathbb{P}})$, and families of random variables $(\widetilde{\mathcal{B}}_{n},\widetilde{U}_{n},\widetilde{U_{n,0}}),$ $(\widetilde{\mathcal{B}},\widetilde{U}, \widetilde{U_{0}})$ defined on $(\widetilde{\Omega}, \widetilde{\mathcal{F}},\widetilde{\mathbb{P}})$ with values in the space $C([0,T],\mathcal{U}_{1})\times \Xi$, such that
\begin{enumerate}
\item[(i)] the law of $(\widetilde{\mathcal{B}}_{n},\widetilde{U}_{n},\widetilde{U_{n,0}})$ is $\mathcal{L}_{n},$ 
\item[(ii)] the law of $(\widetilde{\mathcal{B}},\widetilde{U},\widetilde{U_{0}})$ is $\mathcal{L},$
\item[(iii)] we have $\widetilde{\mathbb{P}}$-a.s.,
\begin{equation}
(\widetilde{\mathcal{B}}_{n},\widetilde{U}_{n},\widetilde{U_{n,0}}) \rightarrow (\widetilde{\mathcal{B}},\widetilde{U},\widetilde{U_{0}}) \quad \text{on} \phantom{x} C([0,T],\mathcal{U}_{1})\times \Xi, \label{convstep1}
\end{equation}
\item[(iv)] $\widetilde{\mathcal{B}}_{n}(\widetilde{\omega}) = \widetilde{\mathcal{B}}(\widetilde{\omega})$, for all $\widetilde{\omega} \in \widetilde{\Omega}$.
\end{enumerate}

Denote by $\widetilde{\mathbb{E}}$ the mathematical expectation under the probability measure $\widetilde{\mathbb{P}}$. Now, consider the augmentation of the following filtration $\{\bar{\mathcal{F}}_{t}\}_{t \in [0,T]}$,
\begin{equation*}
\bar{\mathcal{F}}_{t}:= \sigma\bigl(\widetilde{U}(s), \widetilde{\mathcal{B}}(s),  0\leq s\leq t\bigr),
\end{equation*}
that will be denoted $\{\widetilde{\mathcal{F}}_t\}:=\{\widetilde{\mathcal{F}}_{t}\}_{t \in [0,T]}$. In order to ensure that the stochastic process $(\widetilde{\mathcal{B}}(t))_{t \in [0,T]}$ is an $\{\widetilde{\mathcal{F}}_{t}\}$-Wiener process and that the stochastic integral is well-defined, we will check the axioms of the Lévy's theorem \cite[Proposition 3.11]{da2014stochastic}, namely, $(\widetilde{\mathcal{B}}(t))_{t \in [0,T]}$ is an $\{\widetilde{\mathcal{F}}_{t}\}$-martingale and $\langle \widetilde{\mathcal{B}}, \widetilde{\mathcal{B}}\rangle_{t} = tQ,$ for all $t \in [0,T],$ where $Q$ is the covariance operator of $\mathcal{B}$ (see Subsection \ref{Subsection2.2}).

It is obvious that $\widetilde{\mathcal{B}}$ is $\{\widetilde{\mathcal{F}}_{t}\}$-adapted and $\widetilde{\mathbb{P}}$-a.s. continuous. 
 Moreover, by (iv), the equality of law (i) and since $(\mathcal{B}(t))_{t \in [0,T]}$ is an $\{\mathcal{F}_{t}\}$-Wiener process, we obtain
\begin{equation*}
\widetilde{\mathbb{E}}\vert\widetilde{\mathcal{B}}(0)\vert = \widetilde{\mathbb{E}}\vert\widetilde{\mathcal{B}}_{n}(0)\vert =\mathbb{E}\vert \mathcal{B}(0)\vert = 0,
\end{equation*}
then, $\widetilde{\mathcal{B}}(0) = 0$ a.s. in $\widetilde{\Omega}$.\\
Consider a bounded and continuous functional $\upsilon \in C_{b}(C([0,s],\mathcal{U}_{1}) \times C([0,s],H))$, $0\leq s \leq t \leq T$. For any $\mathfrak{b} \in \mathcal{U}_{1}$, thanks to (iv) and (i), we write 
\begin{align*}
\widetilde{\mathbb{E}}[(\widetilde{\mathcal{B}}(t)-\widetilde{\mathcal{B}}(s),\mathfrak{b})_{\mathcal{U}_{1}}\upsilon(\widetilde{\mathcal{B}},\widetilde{U})] &= \lim_{n \to \infty}\widetilde{\mathbb{E}}[(\widetilde{\mathcal{B}}(t)-\widetilde{\mathcal{B}}(s),\mathfrak{b})_{\mathcal{U}_1}\upsilon(\widetilde{\mathcal{B}},\widetilde{U}_{n})]
= \lim_{n \to \infty}\mathbb{E}[(\mathcal{B}(t)-\mathcal{B}(s),\mathfrak{b})_{\mathcal{U}_1}\upsilon(\mathcal{B},U_{n})] = 0.
\end{align*}
Notice that in the last equality we used the fact that the stochastic process $U_{n}$ is $\{\mathcal{F}_{t}\}$-adapted and $\mathcal{B}$ is an $\{\mathcal{F}_{t}\}$-martingale.
Thus, $\widetilde{\mathcal{B}}$ is an $\{\widetilde{\mathcal{F}}_{t}\}$-martingale.\\ 
Now, let $\lbrace e_{i}\rbrace_{i \in \mathbb{N}}$ be an orthonormal basis of $\mathcal{U}_{1}$. For all $t,s \in [0,T]$, $0\leq s \leq t$ and $l,m \in \mathbb{N}$, again owing to (iv) and (i), we find
\begin{align*}
\widetilde{\mathbb{E}}&[((\widetilde{\mathcal{B}}(t),e_{l})_{\mathcal{U}_1}(\widetilde{\mathcal{B}}(t),e_{m})_{\mathcal{U}_1}- (\widetilde{\mathcal{B}}(s),e_{l})_{\mathcal{U}_{1}}(\widetilde{\mathcal{B}}(s),e_{m})_{\mathcal{U}_{1}}-((t-s)Qe_{l},e_{m}))\upsilon(\widetilde{\mathcal{B}},\widetilde{U})] \notag\\
&= \lim_{n \to \infty}\widetilde{\mathbb{E}}[((\widetilde{\mathcal{B}}(t),e_{l})_{\mathcal{U}_{1}}(\widetilde{\mathcal{B}}(t),e_{m})_{\mathcal{U}_{1}} -(\widetilde{\mathcal{B}}(s),e_{l})_{\mathcal{U}_{1}}(\widetilde{\mathcal{B}}(s),e_{m})_{\mathcal{U}_{1}} - ((t-s)Qe_{l},e_{m})) \upsilon(\widetilde{\mathcal{B}},\widetilde{U}_{n})]\notag\\ 
&= \lim_{n\to \infty}\mathbb{E}[((\mathcal{B}(t),e_{l})_{\mathcal{U}_{1}}(\mathcal{B}(t),e_{m})_{\mathcal{U}_{1}}-(\mathcal{B}(s),e_{l})_{\mathcal{U}_{1}}(\mathcal{B}(s),e_{m})_{\mathcal{U}_{1}}-((t-s)Qe_{l},e_{m}))\upsilon(\mathcal{B},U_{n})]= 0,
\end{align*}
where we used similar arguments as above besides the fact that $\langle\mathcal{B},\mathcal{B}\rangle_{t} = tQ$, for all $t \in [0,T]$. It follows by Lévy's theorem \cite{da2014stochastic}, that the stochastic process $\widetilde{\mathcal{B}}(t)$, $t \in [0,T]$ defined on the stochastic basis $\bigl(\widetilde{\Omega},\widetilde{\mathcal{F}},\{\widetilde{\mathcal{F}}_{t}\}, \widetilde{\mathbb{P}})$ is a $Q$-Wiener process with respect to the filtration $\{\widetilde{\mathcal{F}}_{t}\}_{t\in [0,T]}.$ 
\begin{rmq}
From the definition of the filtration $\lbrace\widetilde{\mathcal{F}}_{t}\rbrace_{t\in [0,T]},$ the process $\widetilde{U}(t)$ is $\{\widetilde{\mathcal{F}}_{t}\}$-adapted for any $t \in [0,T]$, thus $\widetilde{U}(0)$ is $\widetilde{\mathcal{F}}_{0}$-measurable.
\end{rmq}

Let $A := A(u).$ For the sake of brevity, set
\begin{align}
g(u) := \mu\Delta u - u\cdot\nabla u  &+ \alpha_{1}\text{div}(u\cdot \nabla A+(\nabla u)^{T}A+A(\nabla u))+\alpha_{2}\text{div}(A^{2}) + \beta\text{div}(\vert A\vert^{2}A).\label{simpl}
\end{align}

Let $A_{n} := A(U_{n})$ (resp. $\widetilde{A}_{n}:=A(\widetilde{U}_{n})$). Since the following equation is  satisfied by $U_{n}$
\begin{align}
\left(U_n(t)-\alpha_{1}\Delta U_{n}(t) ,\text{v}_j \right) & =\left(U_{n}(0)-\alpha_{1}\Delta U_{n}(0) ,\text{v}_j \right)+\int_{0}^{t}\bigl(g(U_{n}), \text{v}_j \bigr)
\,ds+\int_{0}^{t}\left(\Phi(s,U_n),\text{v}_j \right)\,ds
	\notag\\
	&
\quad+\,\int_{0}^{t}\left(\sigma(s,U_{n}),\text{v}_j \right) \,d\mathcal{B}(s),\qquad \text{for} \phantom{x} j=1, \cdots n,
	\label{Uk}
\end{align}
thanks to the axioms (i) and (iv), one can verify that the process $\widetilde{U}_{n}$ satisfies the following equation $\widetilde{\mathbb{P}}$-a.s.
\begin{align}
(\widetilde{U}_n(t)-\alpha_{1}\Delta\widetilde{U}_{n}(t),\text{v}_j) =&(\widetilde{U_{n,0}}-\alpha_{1}\Delta \widetilde{U_{n,0}},\text{v}_j )+\int_{0}^{t}
\bigl(g(\widetilde{U}_{n}), \text{v}_j \bigr)\,ds+\int_{0}^{t}
\left(\Phi(s,\widetilde{U}_n),\text{v}_j \right)\,ds
	\notag\\
	&
+\,\int_{0}^{t}\left(\sigma(s,\widetilde{U}_{n}),\text{v}_j\right) \,d\widetilde{\mathcal{B}}(s),\qquad \text{for} \phantom{x} j=1,\cdots, n.
	\label{barUk}
\end{align}
For a detailed reasoning we refer $e.g.$ to \cite{bensoussan1995stochastic, razafimandimby2010weak}.
\begin{rmq}
It is easy to notice that from the equation \eqref{barUk}, we have $\widetilde{U_{n,0}} = \widetilde{U}_{n}(0).$
\end{rmq}

\subsubsection{Step 2}
Since the equation \eqref{barUk} holds, proceeding as in the proofs of Lemmas \ref{estimate3}, \ref{estimate1} and \ref{lemma.4}, we deduce the existence of a positive constant $C$ independent of $n$, such that
\begin{align}
&\widetilde{\mathbb{E}}\sup_{s\in [0, t]}\|\widetilde{U}_{n}(s)\|_{V}^{2}+\dfrac{\beta}{2}\widetilde{\mathbb{E}}\int_{0}^{t}\|\widetilde{A}_{n}\|_{4}^{4}ds + 2\mu \widetilde{\mathbb{E}}\int_{0}^{t}\|\nabla \widetilde{U}_{n}\|_{2}^{2}ds \leq C (1 + \widetilde{\mathbb{E}}\|\widetilde{U}(0)\|_{V}^{2}),  \label{tildeE.v}
\\ &
\widetilde{\mathbb{E}}\|\widetilde{U}_{n}\|_{L^{4}(0,t;W^{1,4}(\mathcal{D}))}^{4} \leq C (1 + \widetilde{\mathbb{E}}\|\widetilde{U}(0)\|_{V}^{2}), \label{tildeE.14}
\\&
\widetilde{\mathbb{E}}\sup_{s\in [0, t]}\|\widetilde{U}_{n}(s)\|_{V}^{p} \leq C (1 + \widetilde{\mathbb{E}}\|\widetilde{U}(0)\|_{V}^{p}), \label{tildeE.1}
\\& 
\widetilde{\mathbb{E}}\|\widetilde{U}_{n}\|_{\mathcal{C}^{0,\delta}(0,T;H)}\leq C,
\label{tildeHolder}
\end{align}
for all $t \in [0,T]$, where $\delta = \min(\delta_{1},\delta_{2})$, $\delta_{1}\in (0,\frac{1}{4}]$ and $\delta_{2}\in (0,\gamma p-1)$, for $\gamma \in (0,\frac{1}{2})$.\\ 

The estimate \eqref{tildeE.1} implies 
\begin{equation}
\widetilde{\mathbb{E}}\sup_{t \in[0,T]}\|\widetilde{U}_n\|^q_{2}\leq C, \quad\text{for}\quad q\leq p. \label{E.4}
\end{equation}

Recall from (iii) in $Step$ $1$, that $\widetilde{U}_n$ converges to $\widetilde{U}$ strongly in $C([0,T],H),$ $\widetilde{\mathbb{P}}$-a.s., then $\widetilde{U}_n$ converges to $\widetilde{U}$ strongly in $L^{q}(0,T;H),$ $\widetilde{\mathbb{P}}$-a.s. and
\begin{equation}
\|\widetilde{U}_n(t)\|^q_{2}\to \|\widetilde{U}(t)\|^q_{2}\quad \widetilde{\mathbb{P}}-\text{a.s.}, \phantom{x} \forall t \in [0,T],\quad q < \infty.\label{limt}
\end{equation}
Given \eqref{E.4} and \eqref{limt}, Fatou's Lemma yields
\begin{equation}
\widetilde{\mathbb{E}} \sup_{t\in [0,T]}\|\widetilde{U}\|^q_{2}\leq C, \quad\text{for}\quad q\leq p.	\label{E.5}
\end{equation}
Taking into account \eqref{E.4} and \eqref{E.5}, we deduce that for $q<p$, the sequence $\{\|\widetilde{U}_n-\widetilde{U}\|^q_{2}\}_{n}$ is uniformly integrable in $L^1(\widetilde{\Omega})$ and converges to zero in probability. Thus, by the Vitali convergence theorem, we obtain
\begin{equation}
\widetilde{U}_n\to \widetilde{U} \quad \text{strongly in} \quad L^q(\widetilde{\Omega},L^q(0,T;H)) \quad\text{for} \quad  q<p.\label{converge1}
\end{equation}
Moreover, one can extract a subsequence of $\{\widetilde{U}_n\}_{n}$, still denoted $\{\widetilde{U}_n\}_{n},$ such that
\begin{equation}
\widetilde{U}_n(t) \rightarrow \widetilde{U}(t) \quad \text{strongly in}\quad  H, \quad \widetilde{\mathbb{P}}-\text{a.s.,} \quad \forall t\in [0,T].\label{converge2}
\end{equation}

Furthermore, by \eqref{tildeE.1} and \eqref{tildeE.14} we have
\begin{align}
&\widetilde{U}_n \rightharpoonup \widetilde{U}\quad\text{weakly in} \quad  L^p(\widetilde{\Omega}, L^p(0,T; V))\cap L^4(\widetilde{\Omega}, L^{4}(0,T;\mathbf{W}^{1,4}(\mathcal{D}))), \label{converge3}
\\ &
\widetilde{U}_n\rightharpoonup \widetilde{U}\quad\text{weakly-* in} \quad  L^p(\widetilde{\Omega}, L^\infty_{w}(0,T; V)). \label{convw*}
\end{align}
\begin{rmq}
Notice that, since $\widetilde{U}$ belongs to the spaces $L^p(\widetilde{\Omega}, L^\infty_{w}(0,T; V))$ and $L^p(\widetilde{\Omega}, C([0,T], H))$, then $\widetilde{\mathbb{P}}-a.s.$ $\widetilde{U} \in C_{w}([0,T], V))$, where the latter denotes the space of weakly continuous $V$-valued functions. 
\end{rmq}
Taking into account (i) and (iii) in $Step$ $1$, we deduce that 
\begin{equation*}
\widetilde{\mathbb{E}}\|\widetilde{U_{n,0}}\|_{V}^{p} \leq C,
\end{equation*}
that is $\widetilde{U_{n,0}}$ belongs to the space $L^{p}(\widetilde{\Omega},V)$. Thus, by the Vitali convergence theorem  we have
\begin{equation}
\widetilde{U_{n,0}} \rightarrow \widetilde{U_{0}} \quad \text{strongly in} \phantom{x} L^{q}(\widetilde{\Omega},V), \quad \text{for}\quad q< p. 
\end{equation}
Therefore, $\mathcal{L}(\widetilde{U_{n,0}})$ converges to $\mathcal{L}(\widetilde{U_{0}})$ in $V$, and since $U_{n,0}$ converges strongly to $U_{0}$ in $L^{q}(\Omega,V),$ $q\leq p$, then $\mathcal{L}(U_{n,0})$ converges to $\mathcal{L}(U_{0})$ in $V$, and by uniqueness of the limit, we deduce that 
$$\mathcal{L}(\widetilde{U_{0}}) = \mathcal{L}(U_{0}).$$

\subsection{Monotonicity property}

We rewrite the equation \eqref{barUk} in its explicit form 
\begin{align}
\left(\widetilde{U}_n(t)-\alpha_{1}\Delta \widetilde{U}_{n},\text{v}_j \right) & =\left(\widetilde{U_{n,0}}-\alpha_{1}\Delta \widetilde{U_{n,0}},\text{v}_j \right) +\int_{0}^{t}\bigl(\mu \Delta \widetilde{U}_{n} -\widetilde{U}_{n}\cdot \nabla \widetilde{U}_{n}+ \alpha_{1}\text{div}(\widetilde{U}_{n}\cdot\widetilde{A}_{n})
    \notag\\
    &	
+\alpha_{1}\text{div}((\nabla \widetilde{U}_{n})^{T}\widetilde{A}_{n}+\widetilde{A}_{n}(\nabla \widetilde{U}_{n}))+\alpha_{2}\text{div}(\widetilde{A}_{n}^{2})+\beta \text{div}(\vert \widetilde{A}_{n}\vert^{2}\widetilde{A}_{n}), \text{v}_j \bigr)\,ds 
\notag\\
&+\int_{0}^{t}\left(\Phi(s,\widetilde{U}_n),\text{v}_j \right)\,ds+\,\int_{0}^{t}\left(\sigma(s,\widetilde{U}_{n}),\text{v}_j \right) \,d\widetilde{\mathcal{B}}(s),
	\quad \text{for} \phantom{x} j= 1,\cdots, n.
	\label{tilde.U}
\end{align}

Let us now define the non-linear operator $\mathcal{Q} :V\cap \mathbf{W}^{1,4}(\mathcal{D}) \ni u \mapsto \mathcal{Q}(u) \in (V\cap \mathbf{W}^{1,4}(\mathcal{D}))^*$, as follows 
\begin{equation}
	\label{monoton}
\mathcal{Q}(u) = -\mu \Delta u - \alpha_{1}\text{div}((\nabla u)^{T}A+A(\nabla u))-\alpha_{2}\text{div}(A^{2})-\beta \text{div}(\vert A\vert^{2}A).
\end{equation}
An application of the divergence theorem yields, for $u, y \in V\cap \mathbf{W}^{1,4}(\mathcal{D})$
\begin{equation*}
(\mathcal{Q}(u),y)_{(V\cap W^{1,4})^{*},V\cap W^{1,4}} = \int_{\mathcal{D}} \bigl(\mu \nabla u + \alpha_{1}((\nabla u)^{T}A+A(\nabla u)) +\alpha_{2}A^{2}+ \beta \vert A\vert^{2}A\bigr). \nabla y dx.
\end{equation*}
\begin{lem}(\cite[Lemma 4]{busuioc2008steady})
\label{mon}
Assume that \eqref{RESTMON} holds. Then, the operator $\mathcal{Q}$ is monotone in the following sense
\begin{equation}
(\mathcal{Q}(u)-\mathcal{Q}(y), u-y)_{(V\cap W^{1,4})^{*},V\cap W^{1,4}}\geq 0, \quad \text{for all} \quad u, y \in V\cap \mathbf{W}^{1,4}(\mathcal{D}). \label{monotonicity}
\end{equation}
\end{lem}

Therefore, the equation \eqref{tilde.U} becomes of the form
 \begin{align}
\left( \widetilde{U}_n(t)-\alpha_{1}\Delta \widetilde{U}_{n}(t),\text{v}_j \right)  &=\left( \widetilde{U_{n,0}}-\alpha_{1}\Delta \widetilde{U_{n,0}} ,\text{v}_j \right)+\int_{0}^{t}\bigl(-\mathcal{Q}(\widetilde{U}_{n})-\widetilde{U}_{n}\cdot\nabla\widetilde{U}_{n} +\alpha_{1} \text{div}(\widetilde{U}_{n}\cdot\nabla\widetilde{A}_{n}), \text{v}_j \bigr)\,ds
\notag\\ &+\int_{0}^{t}\left(\Phi(s,\widetilde{U}_n),\text{v}_j \right)\,ds+\int_{0}^{t}\left( \sigma(s,\widetilde{U}_{n}),\text{v}_j \right) \,d\widetilde{\mathcal{B}}(s), \quad \text{for} \phantom{x} j=1,\cdots, n.\label{eq.monotone}
\end{align}

\subsection{Proof of Theorem \ref{main.result}}
This subsection consists of the last step in the proof of Theorem \ref{main.result}. Namely, we will pass to the limit the terms of equation \eqref{eq.monotone}. First of all, we regroup the terms of equation \eqref{eq.monotone} into three expressions, and for each expression we will follow an appropriate way to pass to the limit. Set for $j=1, \cdots, n$ and $t \in [0,T]$
\begin{align*}
E_{n} &:= \bigl(\widetilde{U}_n(t)-\alpha_{1}\Delta\widetilde{U}_{n}(t) ,\text{v}_j \bigr) - \bigl(\widetilde{U_{n,0}}-\alpha_{1}\Delta \widetilde{U_{n,0}} ,\text{v}_j \bigr)+\int_{0}^{t} \left(\widetilde{U}_{n}\cdot\nabla\widetilde{U}_{n}- \alpha_{1}\text{div}(\widetilde{U}_{n}\cdot \nabla \widetilde{A}_{n}), \text{v}_j \right)ds,\\
F_{n} &:= \int_{0}^{t}\bigl(\Phi(s,\widetilde{U}_n),\text{v}_j \bigr)\,ds+\int_{0}^{t}\bigl(\sigma(s,\widetilde{U}_{n}),\text{v}_j \bigr) \,d\widetilde{\mathcal{B}}(s),\\
G_{n} &:= \int_{0}^{t} \bigl(\mathcal{Q}(\widetilde{U}_{n}), \text{v}_{j})ds,
\end{align*}
such that $E_{n} - F_{n} + G_{n} = 0$ corresponds to the equation \eqref{eq.monotone}.\vspace{0.1cm}

Let $j = 1, \cdots, n$ and $\mathcal{A}\in \widetilde{\mathcal{F}}$. By using the convergence \eqref{converge2}, we get for any $ t \in [0,T]$
\begin{align}
\int_{\mathcal{A}}\bigl(\widetilde{U}_n(t)-\alpha_{1}\Delta \widetilde{U}_{n}(t) ,\text{v}_{j} \bigr)d\widetilde{\mathbb{P}}&=\int_{\mathcal{A}}\bigl(\widetilde{U}_n(t) ,\text{v}_{j}-\alpha_{1}\Delta (\text{v}_{j}) \bigr)d\widetilde{\mathbb{P}} \notag\\
&\longrightarrow\int_{\mathcal{A}}\bigl(\widetilde{U}(t) ,\text{v}_{j}-\alpha_{1}\Delta  \text{v}_{j} \bigr)d\widetilde{\mathbb{P}} = \int_{\mathcal{A}}\bigl(\widetilde{U}(t)-\alpha_{1}\Delta \widetilde{U}(t) ,\text{v}_{j} \bigr)d\widetilde{\mathbb{P}}.
\label{E1}
\end{align}
The convergence \eqref{convstep1} yields 
\begin{align}
\int_{A}\bigl(\widetilde{U_{n,0}}-\alpha_{1}\Delta \widetilde{U_{n,0}} ,\text{v}_{j} \bigr)d\widetilde{\mathbb{P}}=\int_{\mathcal{A}}\bigl(\widetilde{U_{n,0}} ,\text{v}_{j} \bigr)_V d\widetilde{\mathbb{P}} \longrightarrow \int_{\mathcal{A}}\bigl(\widetilde{U_0} ,\text{v}_{j} \bigr)_{V}d\widetilde{\mathbb{P}} = \int_{\mathcal{A}}\bigl(\widetilde{U_0}-\alpha_{1}\Delta \widetilde{U_0} ,\text{v}_{j} \bigr)d\widetilde{\mathbb{P}}.
\label{E0}
\end{align}
Using the convergences \eqref{converge1} and \eqref{converge3}, we find for all $t\in [0,T]$
\begin{align}
\int_{\mathcal{A}} \int_{0}^{t}\left(\widetilde{U}_{n}\cdot \nabla\widetilde{U}_n,\text{v}_j \right)ds d\widetilde{\mathbb{P}}
\longrightarrow \int_{\mathcal{A}} \int_{0}^{t}\left(\widetilde{U}\cdot \nabla\widetilde{U}, \text{v}_j\right)dsd\widetilde{\mathbb{P}}. 
\label{E2}
\end{align}
On the other hand, we obtain for $d=2,3$
\begin{align*}
(\text{div}&(\widetilde{U}_{n}\cdot \nabla \widetilde{A}_{n}), \text{v}_j) = \int_{\mathcal{D}}\text{div}(\widetilde{U}_{n}\cdot \nabla \widetilde{A}_{n})\cdot \text{v}_j dx
= -\int_{\mathcal{D}}\widetilde{U}_{n}\cdot \nabla \widetilde{A}_{n}\cdot \nabla(\text{v}_j)dx=\int_{\mathcal{D}}\sum_{l,m=1}^{d}\widetilde{U}_{n}^{l} \widetilde{A}_{n}^{m}\cdot \partial_{l}\partial_{m}(\text{v}_j)dx.
\end{align*}
Then, using the convergences \eqref{converge1} and \eqref{converge3}, we deduce for $d=2,3$, and for any $t \in [0,T]$  
\begin{align}
\notag\int_{\mathcal{A}}&\int_{0}^{t}\left(\text{div}(\widetilde{U}_{n}\cdot \nabla \widetilde{A}_{n}), \text{v}_j\right)dsd\widetilde{\mathbb{P}}= \int_{\mathcal{A}}\int_{0}^{t}\int_{\mathcal{D}}\left(\sum_{l,m=1}^{d}\widetilde{U}_{n}^{l} \widetilde{A}_{n}^{m}\cdot \partial_{l}\partial_{m}(\text{v}_j)\right)dxdsd\widetilde{\mathbb{P}}\\ &\longrightarrow \int_{\mathcal{A}}\int_{0}^{t}\int_{\mathcal{D}}\left(\sum_{l,m=1}^{d}\widetilde{U}^{l} \widetilde{A}^{m}\cdot \partial_{l}\partial_{m}\text{v}_j\right)dxdsd\widetilde{\mathbb{P}}=\int_{\mathcal{A}}\int_{0}^{t}\left(\text{div}(\widetilde{U}\cdot \nabla \widetilde{A}), \text{v}_j\right)dsd\widetilde{\mathbb{P}}.\label{E3}
\end{align}
Combining \eqref{E1}, \eqref{E2} and \eqref{E3}, we deduce the convergence of all terms in $E_n$ weakly in $L^{1}(\widetilde{\Omega})$.\vspace{0,1cm}

Taking into account the assumption \eqref{lipschitz1} and the convergence \eqref{converge1}, we find for all $t\in [0,T]$
\begin{equation}
\int_{0}^{t}(\Phi(s,\widetilde{U}_{n}),\text{v}_j) ds \longrightarrow \int_{0}^{t}(\Phi(s,\widetilde{U}),\text{v}_j)ds \quad \text{in} \phantom{x} L^{1}(\widetilde{\Omega}). \label{force}
\end{equation}
In order to pass to the limit the stochastic term, we use the assumption \eqref{lipschitz2} and the convergence \eqref{converge1}, then for all $t\in [0,T]$
\begin{align*}
\widetilde{\mathbb{E}}\left|\int_0^t(\sigma(s, \widetilde{U}_n)-\sigma(s,\widetilde{U}),\text{v}_j)d\widetilde{\mathcal{B}}(s)\right|^{2} &\leq\widetilde{\mathbb{E}}\left(\int_0^t\sum_{k=1}^{\infty}\vert(\sigma^{k}(s, \widetilde{U}_n)-\sigma^{k}(s,\widetilde{U}), \text{v}_j)\vert^{2}ds\right)\\ &\leq 
\widetilde{\mathbb{E}}\left(\int_0^t\sum_{k=1}^{\infty}\|\sigma^{k}(s, \widetilde{U}_n)-\sigma^{k}(s,\widetilde{U})\|_{2}^{2}ds\right)\\ &\leq
C \|\widetilde{U}_{n}-\widetilde{U}\|^{2}_{L^{2}(\widetilde{\Omega},L^{2}(0,T;H))}\longrightarrow 0 \phantom{x} \text{as} \phantom{x} n \to \infty.
\end{align*}
It follows, for all $t \in [0,T]$
\begin{equation}
\int_0^t(\sigma(s, \widetilde{U}_n), \text{v}_j)d\widetilde{\mathcal{B}}(s)\longrightarrow \int_0^t(\sigma(s, \widetilde{U}), \text{v}_j)d\widetilde{\mathcal{B}}(s)\quad\text{in}\quad L^2(\widetilde{\Omega}).\label{diffu}
\end{equation}
From \eqref{force} and \eqref{diffu}, we conclude that all terms in $F_n$ converge in $L^{1}(\widetilde{\Omega})$.\vspace{0.1cm}

At this stage, we shall address the convergence of $G_{n}$, namely, of the monotone part introduced in \eqref{monoton}. Indeed, the reasoning is based on the monotonicity Lemma \ref{mon} (see \cite{paicu2008global}, \cite{busuioc2008steady}), together with Itô's Lemma.\vspace{0.1cm}

Set $\mathcal{E}:= L^2(\widetilde{\Omega}\times [0,T];V)\cap L^4(\widetilde{\Omega}\times [0,T]; \mathbf{W}^{1,4}(\mathcal{D}))$ and $\mathcal{X}:= V\cap \mathbf{W}^{1,4}(\mathcal{D}).$ Thanks to the estimates \eqref{E.v} and \eqref{E.14}, one can prove easily that for any $n \in \mathbb{N}$
\begin{align*}
\vert(\mathcal{Q}(\widetilde{U}_n),Y)_{\mathcal{E}^*,\mathcal{E}}\vert = \left\vert\widetilde{\mathbb{E}}	\int_0^T(\mathcal{Q}(\widetilde{U}_{n}),Y)_{\mathcal{X}^{*},\mathcal{X}}dt\right\vert \leq C \|Y\|_{\mathcal{E}}, \qquad \forall Y \in \mathcal{E}.
\end{align*}
Therefore, the operator $\mathcal{Q}(\widetilde{U}_{n})$ is bounded in $\mathcal{E}^{*}$, then there exists $\widetilde{\mathrm{Q}} \in \mathcal{E}^{*}$ such that up to a subsequence, still denoted the same, we have
\begin{equation}
\mathcal{Q}(\widetilde{U}_{n}) \rightharpoonup \widetilde{\mathrm{Q}} \quad \text{weakly in}\quad \mathcal{E}^{*}.\label{Qconv}
\end{equation}

Now, taking the limit (termwise) weakly in the equation \eqref{eq.monotone}, yields
\begin{align}
\int_{\mathcal{A}}(\widetilde{U}(t)-\alpha_{1}\Delta \widetilde{U}(t),\text{v}_j )d\widetilde{\mathbb{P}} & =\int_{\mathcal{A}}(\widetilde{U_{0}}-\alpha_{1}\Delta \widetilde{U_0},\text{v}_j )d\widetilde{\mathbb{P}} +\int_{\mathcal{A}}\int_{0}^{t} (-\widetilde{\mathrm{Q}}-\widetilde{U}\cdot\nabla\widetilde{U}+ \alpha_{1}\text{div}(\widetilde{U}\cdot\nabla\widetilde{A})
\notag\\
&
+\Phi(s,\widetilde{U}),\text{v}_j)\,dsd\widetilde{\mathbb{P}} +\int_{\mathcal{A}}\int_{0}^{t}(\sigma(s,\widetilde{U}),\text{v}_j ) \,d\widetilde{\mathcal{B}}_s d\widetilde{\mathbb{P}},\quad \forall j\in \mathbb{N}.\label{conv.monotone}
\end{align}
\begin{rmq}
It is straightforward to notice that from the equation \eqref{conv.monotone}, we have $\widetilde{U_{0}}= \widetilde{U}(0)$.
\end{rmq}

It remains to show the identification $\widetilde{\mathrm{Q}} = \mathcal{Q}(\widetilde{U})$. In fact, we will prove that 
\begin{equation}
\widetilde{\mathbb{E}}\int_0^T(\widetilde{\mathrm{Q}}-\mathcal{Q}(Y),\widetilde{U}-Y)_{\mathcal{X}^{*},\mathcal{X}}dt\geq 0, \phantom{xxx} \forall Y \in \mathcal{E}. \label{Qmon}
\end{equation}
First of all, assume that the inequality \eqref{Qmon} is true and choose $Y = \widetilde{U} -\varepsilon Z$, for an arbitrary $Z \in \mathcal{E}$, $\varepsilon >0$, then
\begin{equation*}
\widetilde{\mathbb{E}}\int_0^T(\widetilde{\mathrm{Q}}-\mathcal{Q}(\widetilde{U}-\varepsilon Z), Z)_{\mathcal{X}^{*},\mathcal{X}}dt\geq 0.
\end{equation*}
Since $\mathcal{Q}$ is hemicontinuous, let $\varepsilon$ tends to $0$, we obtain
\begin{equation*}
\widetilde{\mathbb{E}}\int_0^T(\widetilde{\mathrm{Q}} - \mathcal{Q}(\widetilde{U}),Z)_{\mathcal{X}^{*},\mathcal{X}}dt\geq 0.
\end{equation*}
The arbitrary choice of $Z$ yields $\widetilde{\mathrm{Q}} = \mathcal{Q}(\widetilde{U})$.\vspace{0.15cm}

Now, let us prove that \eqref{Qmon} holds. We have
\begin{align*}
(\widetilde{\mathrm{Q}}-\mathcal{Q}(Y),\widetilde{U}-Y)_{\mathcal{X}^{*},\mathcal{X}} &= (\mathcal{Q}(\widetilde{U}_{n})-\mathcal{Q}(Y),\widetilde{U}_{n}-Y)_{\mathcal{X}^{*},\mathcal{X}} + (\mathcal{Q}(\widetilde{U}_{n})-\widetilde{\mathrm{Q}},Y)_{\mathcal{X}^{*},\mathcal{X}}\\ &+ (\mathcal{Q}(Y),\widetilde{U}_{n}-\widetilde{U})_{\mathcal{X}^{*},\mathcal{X}} + (\widetilde{\mathrm{Q}},\widetilde{U})_{\mathcal{X}^{*},\mathcal{X}}-(\mathcal{Q}(\widetilde{U}_{n}),\widetilde{U}_{n})_{\mathcal{X}^{*},\mathcal{X}}= I_{1} + I_{2} + I_{3}+ I_{4}.
\end{align*}
By monotonicity of $\mathcal{Q}$, we have 
$\widetilde{\mathbb{E}} \int_{0}^{T} I_{1}dt= \widetilde{\mathbb{E}} \int_{0}^{T} (\mathcal{Q}(\widetilde{U}_{n})-\mathcal{Q}(Y),\widetilde{U}_{n}-Y)_{\mathcal{X}^{*},\mathcal{X}}dt \geq 0.$ On the other hand, \eqref{Qconv} implies that $\widetilde{\mathbb{E}} \int_{0}^{T}I_{2} dt= \widetilde{\mathbb{E}} \int_{0}^{T}(\mathcal{Q}(\widetilde{U}_{n})-\widetilde{\mathrm{Q}},Y)_{\mathcal{X}^{*},\mathcal{X}}dt \rightarrow 0,$ \text{as} $n \to \infty.$ The convergence \eqref{converge3} yields $\widetilde{\mathbb{E}} \int_{0}^{T} I_{3}dt= \widetilde{\mathbb{E}} \int_{0}^{T}(\mathcal{Q}(Y),\widetilde{U}_{n}-\widetilde{U})_{\mathcal{X}^{*},\mathcal{X}}dt \rightarrow 0,$ \text{as} $n \to \infty.$ It remains to verify the sign of $\widetilde{\mathbb{E}} \int_{0}^{T}I_{4}dt = \widetilde{\mathbb{E}} \int_{0}^{T}[(\widetilde{\mathrm{Q}},\widetilde{U})_{\mathcal{X}^{*},\mathcal{X}}-(\mathcal{Q}(\widetilde{U}_{n}),\widetilde{U}_{n})_{\mathcal{X}^{*},\mathcal{X}}]dt$.
\vspace*{0.5cm}
Recall the equation \eqref{eq.monotone}, then applying Itô's formula for the function $x \mapsto x^{2}$, we derive
\begin{align*}
d(\widetilde{U}_{n},\text{v}_{j})^{2}_{V}= \bigl(&-2(\widetilde{U}_{n},\text{v}_{j})_{V}(\mathcal{Q}(\widetilde{U}_{n}),\text{v}_{j})-2(\widetilde{U}_{n},\text{v}_{j})_{V}(\widetilde{U}_{n}\cdot\nabla \widetilde{U}_{n},\text{v}_{j})\notag\\ & + 2(\widetilde{U}_{n},\text{v}_{j})_{V}(\alpha_{1}\text{div}(\widetilde{U}_{n}\cdot \nabla\widetilde{A}_{n}),\text{v}_{j}) 
+2(\widetilde{U}_{n},\text{v}_{j})_{V}(\Phi(\cdot,\widetilde{U}_{n}),\text{v}_{j})\notag\\ & + \sum_{k=1}^{\infty}\vert(\sigma^{k}(\cdot,\widetilde{U}_{n}), \text{v}_{j})\vert^{2}\bigr)dt + (\widetilde{U}_{n},\text{v}_{j})_{V}(\sigma(\cdot,\widetilde{U}_{n}),\text{v}_{j})d\widetilde{\mathcal{B}}(t), \qquad \forall j=1,\cdots, n.
\end{align*}
Multiplying by $\frac{1}{\lambda_{j}}$ and summing over $j=1,\cdots, n$, we get
\begin{align*}
d\|\widetilde{U}_{n}\|_{V}^{2} = (-2(\mathcal{Q}(\widetilde{U}_{n}),\widetilde{U}_{n})_{\mathcal{X}^*,\mathcal{X}} &-2 (\widetilde{U}_{n}\cdot\nabla \widetilde{U}_{n},\widetilde{U}_{n})+2(\alpha_{1}\text{div}(\widetilde{U}_{n}\cdot\nabla \widetilde{A}_{n}),\widetilde{U}_{n}) + 2 (\Phi(\cdot,\widetilde{U}_{n}),\widetilde{U}_{n})\\
& + \sum_{k=1}^{\infty}\sum_{j=1}^n\frac{1}{\lambda_{j}}\vert(\sigma^{k}(\cdot,\widetilde{U}_{n}),\text{v}_{j})\vert^{2})dt + 2(\sigma(\cdot,\widetilde{U}_{n}),\widetilde{U}_{n})d\widetilde{\mathcal{B}}(t).
\end{align*}
Let $\widetilde{\sigma}_{n}^{k},$ $k\in \mathbb{N}^+,$ denote the solution to the modified Stokes problem \eqref{Stokes1} for $f =\sigma^{k}(\cdot,\widetilde{U}_{n}),$ $k \in \mathbb{N}^{+}$. Taking into account \eqref{Stokes1eq}, we obtain
\begin{align*}
\sum_{k=1}^{\infty}\sum_{j=1}^n\frac{1}{\lambda_{j}}\vert(\sigma^{k}(\cdot,\widetilde{U}_{n}),\text{v}_{j})\vert^{2} &= \sum_{k=1}^{\infty}\sum_{j=1}^n\frac{1}{\lambda_{j}}\vert(\widetilde{\sigma}_{n}^{k},\text{v}_{j})_{V}\vert^{2} = \sum_{k=1}^{\infty}\|\widetilde{\sigma}_{n}^{k}\|_{V}^{2}.
\end{align*}
Then, the above equation becomes of the following form
\begin{align}
d\|\widetilde{U}_{n}\|_{V}^{2} = (-2(\mathcal{Q}(\widetilde{U}_{n}),\widetilde{U}_{n})_{\mathcal{X}^*,\mathcal{X}} &-2 (\widetilde{U}_{n}\cdot\nabla \widetilde{U}_{n},\widetilde{U}_{n})+2(\alpha_{1}\text{div}(\widetilde{U}_{n}\cdot\nabla \widetilde{A}_{n}),\widetilde{U}_{n}) + 2 (\Phi(\cdot,\widetilde{U}_{n}),\widetilde{U}_{n})\notag \\
& + \sum_{k=1}^{\infty}\|\widetilde{\sigma}_{n}^{k}\|_{V}^{2})dt + 2(\sigma(\cdot,\widetilde{U}_{n}),\widetilde{U}_{n})d\widetilde{\mathcal{B}}(t).\label{Ito}
\end{align}
Notice that by the anti-symmetry of the trilinear form \eqref{trilin2}
\begin{align}
(\widetilde{U}_{n}\cdot\nabla \widetilde{U}_{n},\widetilde{U}_{n}) = b(\widetilde{U}_{n}, \widetilde{U}_{n}, \widetilde{U}_{n})  = -b(\widetilde{U}_{n}, \widetilde{U}_{n}, \widetilde{U}_{n})= 0. \label{eq0}
\end{align}
In addition, using the divergence theorem and the symmetry of the matrix $\widetilde{A}_n$, we obtain
\begin{align}
(\text{div}(\widetilde{U}_{n}\cdot\nabla \widetilde{A}_{n}),\widetilde{U}_{n})
= -\int_{\mathcal{D}}\widetilde{U}_{n}\cdot\nabla \widetilde{A}_{n}\cdot\nabla\widetilde{U}_{n}dx = -\dfrac{1}{2}\int_{\mathcal{D}}\widetilde{U}_{n}\cdot\nabla \widetilde{A}_{n} \cdot\widetilde{A}_{n}dx=0.\label{eq00}
\end{align}
Integrating the equation \eqref{Ito} over the time interval $[0,T],$ applying the expectation, and taking into account \eqref{eq0}, \eqref{eq00}, we obtain
\begin{align*}
\widetilde{\mathbb{E}}\|\widetilde{U}_{n}(T)\|_{V}^{2}-\widetilde{\mathbb{E}}\|\widetilde{U}_{n}(0)\|_{V}^{2} = &-2\widetilde{\mathbb{E}}\int_{0}^{T}(\mathcal{Q}(\widetilde{U}_{n}),\widetilde{U}_{n})_{\mathcal{X}^*,\mathcal{X}}dt  + 2\widetilde{\mathbb{E}}\int_{0}^{T} (\Phi(t,\widetilde{U}_{n}),\widetilde{U}_{n})dt +\widetilde{\mathbb{E}} \int_{0}^{T}\sum_{k=1}^{\infty}\|\widetilde{\sigma}^{k}_{n}\|_{V}^{2}dt.
\end{align*}
Thus 
\begin{align}
\notag \widetilde{\mathbb{E}}\int_{0}^{T}(\mathcal{Q}(\widetilde{U}_{n}),\widetilde{U}_{n})_{\mathcal{X}^*,\mathcal{X}}dt =& -\dfrac{1}{2}(\widetilde{\mathbb{E}}\|\widetilde{U}_{n}(T)\|_{V}^{2}-\widetilde{\mathbb{E}}\|\widetilde{U}_{n}(0)\|_{V}^{2}) \\ &+ \widetilde{\mathbb{E}}\int_{0}^{T} (\Phi(t,\widetilde{U}_{n}),\widetilde{U}_{n})dt +\dfrac{1}{2}\widetilde{\mathbb{E}} \int_{0}^{T}\sum_{k=1}^{\infty}\|\widetilde{\sigma}_{n}^{k}\|_{V}^{2}dt.\label{Q(U)}
\end{align}
Recall now the equation \eqref{conv.monotone}, then proceeding a similar reasoning as above with the function $(\widetilde{U}(t),\text{v}_j)_V^2$, we infer for all $j \in \mathbb{N}$
\begin{align}
\notag\widetilde{\mathbb{E}}\int_{0}^{T}(\widetilde{\mathrm{Q}},\widetilde{U})_{\mathcal{X}^*,\mathcal{X}}dt = &-\dfrac{1}{2}(\widetilde{\mathbb{E}}\|\widetilde{U}(T)\|_{V}^{2}-\widetilde{\mathbb{E}}\|\widetilde{U}(0)\|_{V}^{2}) \\ &+ \widetilde{\mathbb{E}}\int_{0}^{T} (\Phi(t,\widetilde{U}),\widetilde{U})dt +\dfrac{1}{2}\widetilde{\mathbb{E}} \int_{0}^{T}\sum_{k=1}^{\infty}\|\widetilde{\sigma}^{k}\|_{V}^{2}dt.\label{Q}
\end{align}
where $\widetilde{\sigma}^{k},$ $k\in \mathbb{N}^+,$ denotes the solution to the modified Stokes problem \eqref{Stokes1}, for $f=\sigma^{k}(t,\widetilde{U}),$ $k\in \mathbb{N}^+$. Taking the difference of the equalities \eqref{Q(U)} and \eqref{Q}, we obtain
\begin{align*}
\widetilde{\mathbb{E}}\int_{0}^{T}\left[(\widetilde{\mathrm{Q}},\widetilde{U})_{\mathcal{X}^*,\mathcal{X}}-(\mathcal{Q}(\widetilde{U}_{n}),\widetilde{U}_{n})_{\mathcal{X}^*,\mathcal{X}}\right]dt &= \dfrac{1}{2}\widetilde{\mathbb{E}}\|\widetilde{U}_{n}(T)\|_{V}^{2}-\dfrac{1}{2}\widetilde{\mathbb{E}}\|\widetilde{U}(T)\|_{V}^{2} +\dfrac{1}{2}\widetilde{\mathbb{E}}\|\widetilde{U}(0)\|_{V}^{2}-\dfrac{1}{2}\widetilde{\mathbb{E}}\|\widetilde{U}_{n}(0)\|_{V}^{2}  \notag\\ &
 + \widetilde{\mathbb{E}}\int_{0}^{T} (\Phi(t,\widetilde{U}),\widetilde{U})dt-\widetilde{\mathbb{E}}\int_{0}^{T} (\Phi(t,\widetilde{U}_{n}),\widetilde{U}_{n})dt\notag\\ &
  + \dfrac{1}{2}\widetilde{\mathbb{E}} \int_{0}^{T}\sum_{k=1}^{\infty}\|\widetilde{\sigma}^{k}\|_{V}^{2}dt-\dfrac{1}{2}\widetilde{\mathbb{E}} \int_{0}^{T}\sum_{k=1}^{\infty}\|\widetilde{\sigma}_{n}^{k}\|_{V}^{2}dt.
\end{align*}
Moreover, using Bessel's inequality, we get
\begin{align*}
\widetilde{\mathbb{E}}\int_{0}^{T}\left[(\widetilde{\mathrm{Q}},\widetilde{U})_{\mathcal{X}^*,\mathcal{X}}-(\mathcal{Q}(\widetilde{U}_{n}),\widetilde{U}_{n})_{\mathcal{X}^*,\mathcal{X}}\right]dt &\geq \dfrac{1}{2}\widetilde{\mathbb{E}}\|\widetilde{U}_{n}(T)\|_{V}^{2}-\dfrac{1}{2}\widetilde{\mathbb{E}}\|\widetilde{U}(T)\|_{V}^{2} \notag\\
& + \widetilde{\mathbb{E}}\int_{0}^{T} (\Phi(t,\widetilde{U}),\widetilde{U})dt -\widetilde{\mathbb{E}}\int_{0}^{T} (\Phi(t,\widetilde{U}_{n}),\widetilde{U}_{n})dt
\notag\\&  + \dfrac{1}{2}\widetilde{\mathbb{E}} \int_{0}^{T}\sum_{k=1}^{\infty}\|\widetilde{\sigma}^{k}\|_{V}^{2}dt-\dfrac{1}{2}\widetilde{\mathbb{E}} \int_{0}^{T}\sum_{k=1}^{\infty}\|\widetilde{\sigma}_{n}^{k}\|_{V}^{2}dt.
\end{align*}
Let	$t\in[0,T]$, thanks to \eqref{tildeE.v} we know that $\widetilde{U}_{n}(t)$ is	bounded	in $L^2(\widetilde{\Omega},V)$. Hence, there	 exists	$u_t\in	L^2(\widetilde{\Omega},V)$ such that	
\begin{equation*}
\widetilde{U}_{n}(t)	\rightharpoonup	u_t \quad	\text{in} \phantom{x}	L^2(\widetilde{\Omega},V).
\end{equation*}
Moreover, we know that $\widetilde{U}_{n}(t)$ converges to	$\widetilde{U}(t)$ in $L^2(\widetilde{\Omega},H)$. Therefore, since $V\hookrightarrow	H$,	we deduce that $u_t=\widetilde{U}(t)$. Thus
\begin{equation}
\liminf_{n\to \infty} \widetilde{\mathbb{E}}\|\widetilde{U}_{n}(T)\|_{V}^{2} \geq \widetilde{\mathbb{E}}\|\widetilde{U}(T)\|_{V}^{2}. \label{LiminfT}
\end{equation}
From \eqref{converge1}, we obtain
\begin{equation}
\widetilde{\mathbb{E}}\int_{0}^{T}(\Phi(t,\widetilde{U}_{n}),\widetilde{U}_{n}) \rightarrow \widetilde{\mathbb{E}}\int_{0}^{T} (\Phi(t,\widetilde{U}),\widetilde{U})dt, \phantom{xx} \text{as} \phantom{x} n \to \infty.\label{PHI}
\end{equation}
In view of \eqref{Stokes1ineq}, we have 
\begin{equation}
\|\widetilde{\sigma}_{n}^{k}\|_{V} \leq C \|\sigma^{k}(\cdot,\widetilde{U}_{n})\|_{2} \phantom{x} \text{and} \phantom{x} \|\widetilde{\sigma}^{k}\|_{V} \leq C \|\sigma^{k}(\cdot,\widetilde{U})\|_{2}, \quad \forall k \in \mathbb{N}^{+}.\label{Sigmaineq}
\end{equation}
Therefore, \eqref{Sigmaineq}, \eqref{growth2}, \eqref{tildeE.1} and \eqref{converge3} imply, for all $k \in\mathbb{N}^{+}$ 
\begin{align}
&\widetilde{\mathbb{E}}\int_{0}^{T}\sum_{k=1}^{\infty}\left\vert\|\widetilde{\sigma}^{k}_{n}\|_{V}^{2}-\|\widetilde{\sigma}^{k}\|_{V}^{2}\right\vert dt \leq \widetilde{\mathbb{E}}\int_{0}^{T}\sum_{k=1}^{\infty}\left\vert\|\widetilde{\sigma}^{k}_{n}\|_{V}-\|\widetilde{\sigma}^{k}\|_{V}\right\vert(\|\widetilde{\sigma}^{k}_{n}\|_{V}+\|\widetilde{\sigma}^{k}\|_{V})dt\notag\\ &
\leq  \int_{0}^{T}\sum_{k=1}^{\infty}\widetilde{\mathbb{E}}\|\widetilde{\sigma}^{k}_n-\widetilde{\sigma}^{k}\|_{V} (\|\widetilde{\sigma}^{k}_{n}\|_{V}+\|\widetilde{\sigma}^{k}\|_{V}) dt
 \leq C\left(\int_{0}^{T}\sum_{k=1}^{\infty}\widetilde{\mathbb{E}}\|\widetilde{\sigma}^{k}_{n}-\widetilde{\sigma}^{k}\|_{V}^{2}dt\right)^{\frac{1}{2}} \longrightarrow 0 \quad \text{as} \quad n \to \infty. \label{Sigmaconv}
\end{align}
Taking into account \eqref{LiminfT}, \eqref{PHI} and \eqref{Sigmaconv}, we finally deduce
\begin{equation*}
\liminf_{n \to \infty}\widetilde{\mathbb{E}}\int_{0}^{T}\left[(\widetilde{\mathrm{Q}},\widetilde{U})_{\mathcal{X}^*,\mathcal{X}}-(\mathcal{Q}(\widetilde{U}_{n}),\widetilde{U}_{n})_{\mathcal{X}^*,\mathcal{X}}\right]dt \geq 0,
\end{equation*}
from which follows that $\liminf_{n\to \infty}\widetilde{\mathbb{E}} \int_{0}^{T} I_{4}dt\geq 0$, thus the inequality \eqref{Qmon} holds, and
\begin{equation}
G_n = \int_{0}^{t} (\mathcal{Q}(\widetilde{U}_{n}),\text{v}_{j})ds \rightharpoonup \int_{0}^{t}(\mathcal{Q}(\widetilde{U}),\text{v}_{j})ds, \label{G.conv}
\end{equation}
weakly in $L^{1}(\widetilde{\Omega}).$ Consequently, replacing $\widetilde{\mathrm{Q}}$ by $\mathcal{Q}(\widetilde{U})$ in the equation \eqref{conv.monotone} implies that the stochastic process $\widetilde{U}$ satisfies equation \eqref{weak.stoch} for $\varphi = \text{v}_{j},$ $\forall j \in \mathbb{N}$. However, span$\overline{\{\text{v}_{1},\cdot \cdot \cdot, \text{v}_{n},\cdot \cdot \cdot\}}^{V}=V,$ hence due to regularity reasons, $\widetilde{U}$ satisfies equation \eqref{weak.stoch} for all $\varphi \in V \cap H^{2}$. Thereby, the proof of Theorem \ref{main.result} concludes, and the system $(\widetilde{\Omega}, \widetilde{\mathcal{F}}, \widetilde{\mathbb{P}}, \{\mathcal{F}_{t}\}_{t\in [0,T]},\widetilde{\mathcal{B}},\widetilde{U})$ is a stochastic weak solution to the enhanced equation \eqref{syst} in the sense of Definition \ref{weak.sol}.

\begin{center}
\section{Long time behaviour of martingale solutions}\label{section4}
\end{center}
\setcounter{equation}{0}

In this section, we study the asymptotic behaviour of martingale solutions to stochastic third grade fluids equations of type \eqref{syst}, namely, we aim to prove that $H^1$-weak solutions $U(t)$ to the system \eqref{syst} converge to the stationary trivial solution $u^\infty := 0$, as $t \to \infty$. Particularly, we will prove the moment exponential stability and the almost sure exponential stability according to stability concepts studied in, $e.g.$ \cite[Definitions 2.2, 2.3]{caraballo2002exponential}, \cite[Definitions 5.1, 5.2]{cipriano2019asymptotic}.

\begin{hyp}
The force and the diffusion coefficient satisfy the following conditions
\begin{align}
&\|\Phi(t,U)\|_{2}^2\leq c_\Phi C_\ell e^{-\eta_1 t}(1+ \|U\|_{V}^2), \qquad \forall t \in [0,\infty) \label{ForceSTAT}
\\ &
\|\sigma(t,U)\|_{L_{2}^0}^2\leq C_\ell e^{-\eta_1 t}(1+ \|U\|_{V}^2),\qquad \forall t \in [0,\infty) \label{DIFSTAT}
\end{align}
where $c_\Phi,$ $C_\ell$ and $\eta_1$ are positive constants.
\end{hyp}
\subsection{The moment exponential stability}
\begin{thm}
Assume that the hypotheses \eqref{ForceSTAT}, \eqref{DIFSTAT} hold and that 
\begin{equation}
2\mu > c_\Phi + 2(c_{\alpha_{1}}^2+ c_{\alpha_{2}}^2)/\beta,\label{Largemu}
\end{equation} 
where $c_\Phi,$ $c_{\alpha_1}$ and $c_{\alpha_2}$ are positive constants depending on $\Phi,$ $\alpha_1$ and $\alpha_2,$ respectively. Then, any martingale solution $U(t)$ to the equation \eqref{syst} converges to the stationary trivial solution $u^\infty := 0$ exponentially in the mean square i.e., there exist a positive constant $\Lambda$ and a real number $\eta \in (0,\eta_1),$ such that
\begin{equation*}
\mathbb{E}\|U(t)\|_V^2\leq \Lambda e^{-\eta t}, \qquad \forall t \in [0,\infty).
\end{equation*}\label{EXPSTAB}
\end{thm}
\begin{proof}
Recall the system \eqref{syst} and set 
\begin{align}
&f(U):=\mu \Delta U- U\cdot \nabla U+\alpha_{1}\text{div}(U\cdot \nabla A+(\nabla U)^T A+ A(\nabla U))+\alpha_{2}\text{div}(A^2) + \beta\text{div}(\vert A\vert^2 A)+\Phi(\cdot,U). \label{F(u)}
\end{align}
Denote by $\widetilde{f}$ (resp. $\widetilde{\sigma}^k,$ $k\in \mathbb{N}^+$) the solution to the modified Stokes problem \eqref{Stokes1}, for $f = f(U)$ (resp. $f=\sigma^k(\cdot,U),$ $k \in \mathbb{N}^+$).\vspace{0.1cm}

Applying Itô's formula to the function $e^{\eta t}\|U(t)\|_{V}^{2}$, yields
\begin{align*}
de^{\eta t}\|U(t)\|_{V}^{2} = \eta e^{\eta t}\|U(t)\|_{V}^{2}dt + 2e^{\eta t}(\widetilde{f},U(t))_{V}dt + 2e^{\eta t}\sum_{k=1}^\infty(\widetilde{\sigma}^k,U(t))_{V}d\beta_{k}(t)+ e^{\eta t}\sum_{k=1}^\infty\| \widetilde{\sigma}^k\|^2_Vdt.
\end{align*}
Thanks to \eqref{Stokes1eq}, \eqref{Stokes1ineq}, we obtain
\begin{align*}
de^{\eta t}\|U(t)\|_{V}^{2} \leq \eta e^{\eta t}\|U(t)\|_{V}^{2}dt + 2e^{\eta t}(f(U),U(t))dt + 2e^{\eta t}(\sigma(t,U),U(t))d\mathcal{B}_{t}+ Ce^{\eta t}\|\sigma(t,U)\|^2_{L_{2}^0}dt,
\end{align*}
where we used the following identifications
\begin{equation}
\sum_{k=1}^\infty(\sigma^k(\cdot,U),U)d\beta_k = (\sigma(\cdot,U),U)d\mathcal{B}, \qquad \sum_{k=1}^\infty\|\sigma^k(\cdot,U)\|_{2}^2= \|\sigma(\cdot, U)\|_{L_2^0}^2.\label{IDENTIF}
\end{equation}
Now, integrating over the time interval $[0,t]$ and taking the expectation in the above inequality, we get
\begin{align}
e^{\eta t}\mathbb{E}\|U(t)\|_{V}^{2} &\leq   \mathbb{E}\|U_0\|_{V}^{2}+\eta\int_{0}^{t} e^{\eta s}\mathbb{E}\|U(s)\|_{V}^{2}ds\notag\\ &+2\int_{0}^{t}e^{\eta s}\mathbb{E}(f(U),U(s))ds + C\int_{0}^{t}e^{\eta s}\mathbb{E}\| \sigma(s,U)\|^2_{L_{2}^0}ds.\label{Statineq}
\end{align}

Next, we estimate the terms of the third expression in the $r.h.s.$ of the inequality \eqref{Statineq}.
The divergence theorem yields
\begin{align}
(\mu \Delta U(t), U(t)) &= -\mu( \nabla U(t), \nabla U(t))= -\mu \|\nabla U(t)\|_{2}^{2}.\label{1S}
\end{align}
Using the property \eqref{trilin2}, we obtain
\begin{align}
(U(t)\cdot \nabla U(t), U(t)) &= b(U(t), U(t), U(t))= -b(U(t), U(t), U(t))= 0.\label{4.7}
\end{align}
Applying the divergence theorem multiple times and using the symmetry of the matrix $A,$ we infer
\begin{align}
(\alpha_{1}\text{div}(U\cdot \nabla A),U(t))=  -\alpha_1\int_{\mathcal{D}}(U\cdot \nabla A) \cdot \nabla U(t)dx = -\dfrac{\alpha_1}{2}\int_{\mathcal{D}}(U\cdot \nabla A) \cdot A dx = 0. \label{4.8}
\end{align}
The divergence theorem, the symmetry of the matrix $A$ and Young's inequality, imply
\begin{align}
\vert (\alpha_{1}\text{div}((\nabla U)^{T}A+A(\nabla U)), U(t))\vert &\leq \left\vert\int_{\mathcal{D}}\alpha_{1} ((\nabla U)^{T}A+A(\nabla U))\cdot \nabla U(t)dx\right\vert\notag\\
& \leq \left\vert \dfrac{\alpha_1}{2} \int_{\mathcal{D}}((\nabla U)^T A +A(\nabla U)) \cdot A dx\right\vert  \leq \left\vert \alpha_{1} \int_{\mathcal{D}} A(\nabla U)\cdot A dx\right\vert  \notag \\
&\leq \alpha_{1} \|A\|_{4}^2\|\nabla U\|_{2}  \leq \dfrac{\beta}{4}\|A\|_{4}^{4} + \dfrac{c_{\alpha_{1}}^{2}}{\beta}\|U\|_{V}^2.
\end{align}
Following a similar reasoning as above, we obtain
\begin{align}
\vert(\alpha_2\text{div}(A^2) ,U(t))\vert \leq \left\vert \alpha_2 \int_{\mathcal{D}}A^2\cdot \nabla U(t)dx\right\vert   
\leq \vert \alpha_2\vert \|A\|_{4}^2\|\nabla U(t)\|_{2} \leq \dfrac{\beta}{4}\|A\|_{4}^{4}+\dfrac{c_{\alpha_2}^2}{\beta}\|U\|_{V}^2.
\end{align}
Again, using the divergence theorem together with the symmetry of $A$, yield
\begin{align}
(\beta\text{div}(\vert A\vert^2 A), U(t)) & = -\beta\int_{\mathcal{D}}\vert A\vert^{2}A\cdot \nabla U(t)dx=-\dfrac{\beta}{2}\int_{\mathcal{D}}\vert A\vert^2 A \cdot A dx  = -\dfrac{\beta}{2}\|A\|_{4}^{4}.
\end{align}
By the assumption \eqref{ForceSTAT} and Young's inequality, we have
\begin{align}
\vert(\Phi(t,U),U(t))\vert &\leq \|\Phi(t,U)\|_{2}\|U(t)\|_{2} 
\leq \dfrac{1}{2c_\Phi} \|\Phi(t,U)\|_{2}^2 + \dfrac{c_\Phi}{2} \|U(t)\|_{2}^2 \notag\\ &\leq \frac{1}{2}C_\ell e^{-\eta_1 t} (1+\|U(t)\|_{V}^{2})+ \dfrac{c_\Phi}{2}\|U(t)\|_{V}^2  .\label{2S}
\end{align}

Recalling the inequality \eqref{Statineq} and taking into account the estimates \eqref{1S}-\eqref{2S} and the assumption \eqref{DIFSTAT}, we obtain 
\begin{align*}
e^{\eta t}\mathbb{E}\|U(t)\|_{V}^{2} 
& \leq \mathbb{E}\|U_0\|_{V}^{2} + \int_{0}^t\left(\eta -2\mu + c_\Phi + \dfrac{2(c_{\alpha_1}^2+c_{\alpha_2}^2)}{\beta}\right)e^{\eta s}\mathbb{E}\|U(s)\|_{V}^2 ds 
\notag\\ &
+C_\ell(1+C) \int_{0}^t e^{(\eta -\eta_1) s}ds + C_{\ell}(1+C)\int_{0}^t e^{-\eta_1 s} e^{\eta s} \mathbb{E}\|U(s)\|_{V}^2  ds.
\end{align*}
We have \eqref{Largemu}, then we can choose $\eta> 0$ such that $\eta- 2\mu + c_\Phi + 2(c_{\alpha_1}^2+ c_{\alpha_2}^2)/\beta< 0,$ thus
\begin{equation*}
e^{\eta t}\mathbb{E}\|U(t)\|_{V}^{2} \leq \mathbb{E}\|U_0\|_{V}^{2} + C_\ell\int_{0}^{t}e^{(\eta -\eta_{1})s}ds + C_\ell\int_{0}^{t}e^{-\eta_1 s}e^{\eta s} \mathbb{E}\|U(s)\|_{V}^{2}ds.
\end{equation*}
Applying Gronwall's Lemma, we get
\begin{align*}
e^{\eta t}\mathbb{E}\|U(t)\|_{V}^{2} &\leq  \left(\mathbb{E}\|U_0\|_{V}^{2} + \dfrac{C_\ell}{\eta_1-\eta}(1-e^{(\eta-\eta_{1})t})\right)e^{C_\ell\int_{0}^{t}e^{-\eta_1 s}ds}.
\end{align*}
Then, we deduce the claimed result
\begin{equation}
\mathbb{E}\|U(t)\|_{V}^{2} \leq \Lambda e^{-\eta t}, \qquad \forall t \in [0,\infty)\label{STABILITYmoment}
\end{equation}
where $\Lambda := \left(\mathbb{E}\|U_0\|_{V}^{2} +  \frac{C_\ell}{\eta_1-\eta}\right)e^{C_\ell/\eta_{1}}.$
\end{proof}
\subsection{The almost sure exponential stability}
\begin{thm}
Assume that all conditions in Theorem \ref{EXPSTAB} hold. Then, any martingale solution $U(t)$ to the equation \eqref{syst} converges to the stationary trivial solution $ u^\infty:=0 $ almost surely exponentially. Namely, there exists a positive real number $ \lambda$, such that
\begin{equation*}
\lim_{t\to \infty}\frac{1}{t}\log\|U(t)\|_{V}\leq -\lambda, \qquad \mathbb{P}-\text{a.s.}.
\end{equation*}
\end{thm}
\begin{proof}
Recall \eqref{F(u)} and apply Itô's formula to the function $ \|U(t)\|_{V}^{2} $, we obtain
\begin{align*}
d\|U(t)\|_{V}^{2}& = 2(\widetilde{f},U(t))_{V}dt + 2\sum_{k=1}^\infty(\widetilde{\sigma}^k,U(t))_{V}d\beta_{k}(t)+ \sum_{k=1}^\infty\| \widetilde{\sigma}^k\|^2_V dt,
\end{align*}
where $\widetilde{f}$ (resp. $\widetilde{\sigma}^k,$ $k\in \mathbb{N}^+$) is the solution to the modified Stokes problem \eqref{Stokes1}, for $f = f(U)$ (resp. $f=\sigma^k(\cdot,U),$ $k \in \mathbb{N}^+$). Thanks to \eqref{Stokes1eq}, \eqref{Stokes1ineq} and \eqref{IDENTIF}, we get
\begin{align*}
d\|U(t)\|_{V}^{2} &\leq 2(f(U),U(t))dt + 2(\sigma(t,U),U(t))d\mathcal{B}_{t}+ C\|\sigma(t,U)\|^2_{L_{2}^0}dt.
\end{align*}
Let $ N \in \mathbb{N} $. Integrating over the time interval $ [N,t], $  yields
\begin{align}
\|U(t)\|_{V}^{2} \leq \|U(N)\|_{V}^{2} + 2\int_{N}^{t}(f(U),U(s))ds + 2\int_{N}^{t}(\sigma(s,U),U(s))d\mathcal{B}_{s}+ C\int_{N}^{t}\| \sigma(s,U)\|^2_{L_{2}^0}ds.\label{4.16}
\end{align}
The Burkholder-Davis-Gundy inequality and Young's inequality imply
\begin{align}
&\mathbb{E}\sup_{t \in [N, N+1]}\left\vert \int_{N}^{t} (\sigma(s,U),U(s))d\mathcal{B}_{s} \right\vert  \leq K \mathbb{E}\left(\int_{N}^{N+1} \vert (\sigma(t,U),U(t))\vert^{2}dt\right)^{\frac{1}{2}} \notag\\ & \leq K \mathbb{E}\left( \int_{N}^{N+1}\|\sigma(t,U)\|_{L_{2}^{0}}^{2}\|U(t)\|_{V}^{2}dt\right)^{\frac{1}{2}}
\leq \mathbb{E}\left( \sup_{t \in[N,N+1]}\|U(t)\|_{V}^{2} K^{2}\int_{N}^{N+1}\|\sigma(t,U)\|_{L_{2}^{0}}^{2}dt\right)^{\frac{1}{2}}\notag\\ & \leq\dfrac{1}{4} \mathbb{E} \sup_{t \in[N,N+1]}\|U(t)\|_{V}^{2} + K^{2}\mathbb{E}\int_{N}^{N+1}\|\sigma(t,U)\|_{L_{2}^{0}}^{2}dt. \label{BURKHexp}
\end{align}
Taking the supremum on $ t\in [N,N+1] $, applying the expectation and then substituting \eqref{BURKHexp} in the inequality \eqref{4.16}, give
\begin{align}
\frac{1}{2}\mathbb{E}\sup_{t\in [N,N+1]}\|U(t)\|_{V}^{2} &\leq \mathbb{E}\|U(N)\|_{V}^{2} + 2\int^{N+1}_{N}\mathbb{E}(f(U),U(t))dt+ (2K^{2}+C)\int_{N}^{N+1}\| \sigma(t,U)\|^2_{L_{2}^0}dt.\label{NINEQ}
\end{align}
In addition, taking into account the estimates \eqref{1S}-\eqref{2S} and the assumption \eqref{DIFSTAT}, we infer
\begin{align*}
\dfrac{1}{2}\mathbb{E}\sup_{t\in [N,N+1]}\|U(t)\|_{V}^{2} &\leq \mathbb{E}\|U(N)\|_{V}^{2} + \int_{N}^{N+1} \left(-2\mu + c_\Phi +\dfrac{2(c_{\alpha_1}^2+c_{\alpha_{2}}^2)}{\beta}\right)\mathbb{E}\|U(t)\|_{V}^{2}dt \notag\\
& + C_\ell(1+2K^2+C)\int_{N}^{N+1}e^{-\eta_1 t}dt + C_\ell(1+2K^2+C)\int_{N}^{N+1}e^{-\eta_1 t}\mathbb{E}\|U(t)\|_{V}^{2}dt.
\end{align*}
Thanks to \eqref{Largemu} and \eqref{STABILITYmoment}, we derive for $ \eta\in (0,\eta_{1}) $
\begin{align*}
\dfrac{1}{2}\mathbb{E}\sup_{t\in [N,N+1]}\|U(t)\|_{V}^{2}
& \leq \Lambda e^{-\eta N} + \frac{C_\ell(2K^{2}+C)}{\eta_1} e^{-\eta_1 N} + C_{\ell}(2K^{2}+C)\Lambda \int_{N}^{N+1} e^{-(\eta+\eta_{1})t} dt.
\end{align*}
Performing simple computations, we obtain for $ \eta \in (0,\eta_{1}) $
\begin{align}
\mathbb{E}\sup_{t\in [N,N+1]}\|U(t)\|_{V}^{2} & \leq \Lambda' e^{-\eta N},
\end{align}
where $ \Lambda' := 2\left(\Lambda + \frac{C_{\ell}(2K^{2}+C)}{\eta_{1}}+\frac{C_{\ell}(2K^{2}+C)\Lambda}{\eta+\eta_{1}}\right). $ The Borel-Cantelli lemma implies the desired result.
\end{proof}

\begin{rmq}\label{F_instead_Phi}
It is worth noting that existence results for martingale solutions in Section \ref{section3} remain valid if we replace the force $\Phi(\cdot, U)$ with an external force $F \in L^2(\Omega \times [0,T]; \mathbf{L}^2(\mathcal{D}))$.
\end{rmq}
\begin{rmq}\label{Remarkregular}
Considering the deterministic system of steady motion type of a third grade fluid contained in a $2$ or $3$-dimensional bounded and simply connected domain $\mathcal{D}$ with a regular boundary $\partial \mathcal{D},$ Viedman \cite{videman1997mathematical}, proved for a sufficiently small external force $h$, the existence and uniqueness of stationary solutions $u^\infty \in \mathbf{W}^{2,d}(\mathcal{D}) \cap V$ in $2d$ and only for a class of small solutions in $3d$. Later on, Bernard and Ouazar \cite{bernard2005stationary}, proved the existence and uniqueness of stationary solutions $u^\infty \in \mathbf{H}^3(\mathcal{D})$ in $2d$ and $3d$ for a sufficiently small force $h \in H(\text{curl}, \mathcal{D})$.
\end{rmq}
\begin{rmq}
Taking into account the existence and uniqueness results of stationary solutions in the aforementioned articles \cite{videman1997mathematical, bernard2005stationary}, we can prove stability results similar to those studied in Section \ref{section4}. More precisely, let $u =(u^{i})_{i=1}^{d},$ $d=2,3,$ denote the velocity vector field and consider the steady system 
\begin{align}
\left\{
\begin{array}{ccc}
- \mu\Delta u + u\cdot \nabla u - \alpha_{1}\mathrm{div}(u \cdot \nabla A(u)+(\nabla u)^{T}A(u) +A(u)(\nabla u))\qquad\quad \vspace{2mm} \\
		\vspace{2mm}
		\quad\quad\qquad\qquad-\alpha_{2} \mathrm{div}(A(u)^{2})-\beta{\rm div}\left(|A(u)|^2 A(u)\right)= h -\nabla p
		&  \multicolumn{1}{l}{\mbox{in}\
			\mathcal{D},} \vspace{2mm}\\
			\multicolumn{1}{l}{\mathrm{div}\,u=0} & \multicolumn{1}{l}{\mbox{in}\
			\mathcal{D},} \vspace{2mm}\\
			\multicolumn{1}{l}{u=0} & \multicolumn{1}{l}{\mbox{on}\
			\partial{\mathcal{D}},}
\end{array}
\right.\label{stationary}
\end{align}
then, under the following regularities of the stationary solution, $u^\infty \in \mathbf{W}^{2,d}(\mathcal{D})\cap V$ or $u^\infty \in \mathbf{H}^3(\mathcal{D})$ (see Remark \ref{Remarkregular}), by imposing appropriate conditions on the exponential decay of the external force $F$ towards $h$ and assuming that the viscosity is sufficiently large, we can show that  martingale solutions $U(t)$ to the system \eqref{syst}, with $F$ instead of $\Phi(\cdot, U)$ (see Remark \ref{F_instead_Phi}), converge to $u^\infty$ exponentially, in the mean square and almost surely, as $t \to \infty$. Moreover, we emphasize that in this case, the almost sure exponential stability relies on the monotonicity property \eqref{monotonicity}.
\end{rmq}

\begin{center}
\section{Conclusion}
\end{center}
Taking an initial condition in the Sobolev space $\mathbf{H}^{1}(\mathcal{D})$, we proved the existence of martingale solutions to the equations of a third grade fluid filling a $2$ or $3$-dimensional bounded and simply connected domain $\mathcal{D}$, with regular boundary $\partial\mathcal{D}$, perturbed by a multiplicative Wiener noise, and supplemented by Dirichlet boundary conditions. The solution is given by a probabilistic system, where the cylindrical Wiener process constitutes a part of it and the velocity vector field is a stochastic process with sample paths in the space $L^{\infty}(0,T;\mathbf{H}^{1}(\mathcal{D}))$.
The adopted strategy is based on the conjugation of both  stochastic compactness criteria involving the tightness property, Prokhorov and Skorokhod theorems, and monotonicity properties of a non-linear operator. In addition, we proved some stability properties of martingale solutions under the condition of a large viscosity.

\centering
\bibliographystyle{abbrv}
\bibliography{3rd-grade-fluids-Raya}

\end{document}